\documentclass[12pt,a4paper,reqno]{amsart}
\usepackage{amsmath,amssymb}
\usepackage[dvipsnames]{xcolor}
\usepackage[colorlinks=true,urlcolor=blue,linkcolor=blue,citecolor=blue]{hyperref}
\usepackage{enumitem}
\usepackage{graphicx}
\usepackage{tikz-cd}
\usetikzlibrary{calc, arrows, decorations.markings}

\numberwithin{equation}{section}
\theoremstyle{plain}
\newtheorem{thm}{Theorem}[section]
\newtheorem{lemma}[thm]{Lemma}
\newtheorem{prop}[thm]{Proposition}
\newtheorem{cor}[thm]{Corollary}

\newcommand{\A}{\mathbb{A}}
\newcommand{\B}{\mathbb{B}}
\newcommand{\C}{\mathbb{C}}
\newcommand{\D}{\mathbb{D}}
\renewcommand{\AA}{\mathbf{A}}
\newcommand{\AB}{\mathbf{A}^{2^\omega}}
\newcommand{\ABxe}{(\mathbf{A}^{2^\omega})^{\mathbf{x}}_{\mathbf{e}}}
\newcommand{\BB}{\mathbf{B}}
\newcommand{\CC}{\mathbf{C}}
\newcommand{\DD}{\mathbf{D}}
\newcommand{\E}{\mathbf{E}}
\newcommand{\F}{\mathcal{F}}
\newcommand{\FF}{\mathbf{F}}

\newcommand{\id}{\mathrm{id}}

\newcommand{\R}{\mathbf{R}}
\newcommand{\N}{\mathbb{N}}
\renewcommand{\S}{\mathbb{S}}

\renewcommand{\P}{\mathbb{P}}
\newcommand{\V}{\mathbf{V}}
\newcommand{\Z}{\mathbb{Z}}
\newcommand{\Q}{\mathbb{Q}}

\renewcommand{\L}{L}
\newcommand{\bL}{\mathbb{L}}
\newcommand{\M}{M}
\newcommand{\bM}{\mathbb{M}}
\newcommand{\MM}{\mathbf{M}}

\newcommand{\ff}{\mathbf{f}}
\newcommand{\bg}{\mathbf{g}}
\newcommand{\hh}{\mathbf{h}}
\newcommand{\ba}{\mathbf{a}}
\newcommand{\bc}{\mathbf{c}}
\newcommand{\bd}{\mathbf{d}}
\newcommand{\be}{\mathbf{e}}
\newcommand{\bo}{\mathbf{0}}
\newcommand{\bx}{\mathbf{x}}
\newcommand{\by}{\mathbf{y}}
\newcommand{\bz}{\mathbf{z}}
\newcommand{\uu}{\mathbf{u}}
\newcommand{\bs}{\mathbf{s}}

\newcommand{\Aut}{\mathrm{Aut}\hspace{0.1em}}

\newcommand{\diag}{\mathrm{diag}\hspace{0.1em}}
\newcommand{\GL}{\mathrm{GL}\hspace{0.1em}}
\newcommand{\PGL}{\mathrm{PGL}\hspace{0.1em}}
\newcommand{\Homeo}{\mathrm{Homeo}\hspace{0.1em}}

\newcommand{\st}{\: :\:}

\newcommand{\dotcup}{\mathbin{\dot{\cup}}}

\newcommand{\restr}{\!\!\restriction}

\title[Ample generics on Boolean powers]{Ample generics in automorphism groups of Boolean powers of simple Mal'cev algebras}
\author{Peter Mayr}
\address{Department of Mathematics,
University of Colorado Boulder, USA}
\email{peter.mayr@colorado.edu}
\author{Nik Ru{\v s}kuc}
\address{School of Mathematics and Statistics,
University of St Andrews, Scotland, UK}  
\email{nik.ruskuc@st-andrews.ac.uk}
\thanks{Supported by the Engineering and Physical Sciences Research Council UKRI2386 and by the  National Science Foundation under Grant No. DMS 2452289}

\date{\today}

\begin{document}
\maketitle

\begin{abstract}
 Let $\AA$ be a finite simple Mal'cev algebra, such as for example a finite simple group, module, ring,
 associative or Lie algebra, loop or quasigroup.
 We show that the automorphism group of a filtered Boolean power of
 continuous functions from the Cantor space $2^\omega$ to $\AA$ has ample generics.
 The proof splits into the abelian and non-abelian cases.
 In the abelian case, we use a representation by modules and the theory of $n$-systems developed by
 Kechris and Rosendal. 
 In the non-abelian case, the proof relies on the decomposition of the automorphism group as a semidirect product of a
 certain closure of a filtered Boolean power of continuous functions from $2^\omega$ to the automorphism group of $\AA$
 and the stabiliser of finitely many points in the homeomorphism group $\Homeo 2^\omega$. 
 As an intermediate step, we show that pointwise stabilisers in $\Homeo 2^\omega$ have ample generics,
 which extends Kwiatkowska's result that $\Homeo 2^\omega$ has ample generics.  
\end{abstract}

 \section{Introduction}
 
 The purpose of this paper is to prove the following.
 
\begin{thm}
\label{thm:DAG}
 Let $\AA$ be a finite simple Mal'cev algebra.
 For $n\geq 0$, let $\be := (e_1,\dots,e_n)$ be a tuple of singleton subalgebras of $\AA$
 and $\bx := (x_1,\dots,x_n)$ a tuple of distinct points in the Cantor space $2^\omega$. 

 Then the automorphism group of the filtered Boolean power
 \[ \ABxe := \{ f\colon 2^\omega\to A \text{ continuous} \st f(x_i)=e_i \text{ for } i=1,\dots,n \} \leq \AB  \]
 has ample generics.
\end{thm}

Here and throughout the paper $\AB$ denotes the algebra of all \emph{continuous} functions from the Cantor space
$2^\omega$ to $\AA$ (with the discrete topology) under pointwise operations.
 
 Theorem~\ref{thm:DAG} applies in particular for $\AA$ a finite simple group, module, ring, associative or Lie algebra,
 loop or quasigroup, etc. (see Section~\ref{sec:algebra} below).
 Its significance is that, in doing so, it establishes a whole host of algebraic structures whose automorphism groups
 have ample generics and consequently the small index property (see Section~\ref{sec:cag} below). 
 To the best of our knowledge, this result is new even for all these classical structures like groups or rings.

 Boolean powers are a useful generalization of direct powers over an infinite index set.
They have been studied as a link between algebra and logic, and also because of their intriguing algebraic and combinatorial behaviour; for more details see Section~\ref{sec:algebra}.
 Filtered Boolean powers $\ABxe$ as in Theorem~\ref{thm:DAG} arise prominently in the variety $V$ generated by $\AA$.
 For concreteness, consider the case of a finite simple non-abelian group $\AA$ with identity $1$ and $x\in 2^\omega$
 in the following.

 First $(\AB)^{x}_{1}$ is the (generalized) Fra{\"i}ss{\'e} limit of the class of finite direct powers
 $\{\AA^n \st n \in\N\}$ by~\cite[Theorem 1.1]{MR:FBP}.
 Secondly $(\AB)^{x}_{1}$ is isomorphic to the smallest normal subgroup $N$ of the (relatively) free group $\FF$
 of countable rank in $V$ such that $\FF/N$ is in the subvariety $W$ generated by all proper subgroups of $\AA$
 (see \cite[p.~367-8]{BE:SIP} or~\cite[p. 201]{BG:ARFG}).
 In fact $\FF$ is a semidirect product of $(\AB)^{x}_{1}$ and the  free group of countable rank in $W$
 \cite[Corollary 4.1]{MR:FBP}.
 Theorem~\ref{thm:DAG} yields that the automorphism group $\Aut(\AB)^x_1$ has ample generics,
 which gives a new proof for the result of Bryant and Evans~\cite{BE:SIP} that $\Aut(\AB)^x_1$ has the small index property.
 As we discuss in~\cite{MR:FBP}, the phenomena mentioned above are not particular to groups but extend to Mal'cev
 algebras (see Section~\ref{sec:algebra}).

\subsection{Outline of the proof of Theorem~\ref{thm:DAG}}
 Yet again, to help the reader focus on the salient points, let us specialise $\AA$ to be a finite simple group.
 The proof is split into two cases, depending on whether $\AA$ is abelian or non-abelian.
 These cases use very different methodologies.

For abelian $\AA$, say $\AA=(\Z_p,+,-,0)$ for $p$ prime, we will prove the result as Theorem~\ref{thm:abag}
in Section~\ref{sec:abelian} using the standard approach introduced by Kechris and Rosendal~\cite{KR:TAG}.
The class $K$ of finite direct powers of $\AA$ (equivalently, the class of finite-dimensional vector spaces over
the field $\FF_p$) has the hereditary property (HP), joint embedding property (JEP) and
the amalgamation property (AP). The Fra{\"i}ss{\'e} limit of $K$ is the countable direct sum $\AA^{(\omega)}$ of $\AA$
(equivalently the countable dimensional vector space over $\FF_p$). 
Every countable subalgebra of a direct power of $\AA$, in particular the Boolean power $\AB$ and the filtered Boolean
power $(\AB)^x_0$ for $x\in 2^\omega$, are isomorphic to $\AA^{(\omega)}$.
We show that the structures of $K$ expanded with $n\in\N$ partial isomorphisms have the JEP and the
weak amalgamation property (WAP). 
By Kechris and Rosendal~\cite[Theorem 6.2]{KR:TAG}, this is equivalent to the  Fra{\"i}ss{\'e} limit of $K$ having generic
$n$-tuples. 
The crucial simplifying step for proving WAP in the abelian case is that every partial isomorphism
on a direct power $\AA^k$ extends to a total automorphism, that is, $K$ has the
extension property for partial automorphisms (EPPA) or \emph{Hrushovski property}.
 
For non-abelian $\AA = (A,\cdot,^{-1},1)$ we have not been able to follow the same approach and to establish the WAP directly.
It is instructive to point out the obstacles and differences to the abelian case:
\begin{enumerate}
\item
 The class $K$ of finite direct powers of $\AA$ has the JEP and AP but is not closed under
 subgroups, hence does not have the HP. Still a (generalized) Fra{\"i}ss{\'e} limit of $K$ exists
 and is isomorphic to $(\AB)^{x}_{1}$ by~\cite[Theorem 1.1]{MR:FBP}.
\item
  The (filtered) Boolean powers $\AB$, $(\AB)^{x}_{1}$ and the countable direct sum $\AA^{(\omega)}$ are all
  non-isomorphic.
 Since $\AB$ does not arise as (generalized) Fra{\"i}ss{\'e} limit, it is not in the scope of the approach
 via the WAP, but is for our Theorem \ref{thm:DAG}.
\item
 $K$ does not have the EPPA, i.e. not
 every partial isomorphism on a direct power $\AA^k$ extends to a total automorphism.
\end{enumerate}  
 To overcome these limitations, we will use a different method to prove Theorem~\ref{thm:DAG} for non-abelian $\AA$
 in Section~\ref{sec:fbp}.
 The basic idea is that by~\cite{MR:FBP} the automorphism group $\Aut\ABxe$ for any (possibly empty) $\bx$ and $\be$
 splits into a semidirect product with complement $H$ isomorphic to the stabilizer of $\bx$ in the group of homeomorphisms
 of $2^\omega$. Kwiat\-kow\-ska~\cite{Kw:GHC} has constructed generics in $\Homeo 2^\omega$
 as projective  Fra{\"i}ss{\'e} limits. Generalizing her approach, we obtain that also the pointwise stabiliser
 $H$ has ample generics in Theorem~\ref{thm:H} below. Then we show that these generics in $H$ are actually generics
 in the semidirect product $\Aut\ABxe$ as well.

 \medskip
 
 The paper is structured as follows. In the remainder of this section we review the necessary notions and context
 for our theorem in more detail. The abelian case is deal with in Section \ref{sec:abelian}, and non-abelian occupies
 Sections \ref{sec:PFT}--\ref{sec:fbp}. We conclude in Section \ref{sec:qus} with some pointers for further investigation.

\subsection{Notation} Throughout the paper we write
\begin{itemize}
\item
 $\N := \{1,2,\dots\}$,
\item
  $[n] := \{1,\dots,n\}$ for $n\in\N$,
\item
 $\Delta_A$ for the diagonal relation $\bigl\{ (a,a)\colon a\in A\bigr\}$ on a set $A$,
\item
 $2^\omega$ for the Cantor space,
\item
  $2^{\omega\circ}$ for the punctured Cantor space $2^\omega\setminus\{ x_1,\dots,x_n\}$  with the subspace topology for
  $n\geq 0$ and distinct $x_1,\dots,x_n\in 2^\omega$
 (up to homeomorphism, $2^{\omega\circ}$ depends only on $n$ and not on the specific choice of $x_1,\dots, x_n$),
\item
 $\Homeo 2^\omega$ for the group of homeomorphisms of $2^\omega$, 
\item
 $\AA^{2^\omega}$ for the algebra of continuous functions from $2^\omega$ to a finite algebra $\AA$ (with discrete
 topology) under the pointwise action of the operations of $\AA$,
\item 
 $\Aut\AA$ for the group of automorphisms of a structure $\AA$.
\item
 For a group $G$ acting on a set $X$, let
  \[ G_{(x_1,\dots,x_n)} := \{ g\in G \st g(x_i) = x_i \text{ for all }i \in [n] \}  \]
 denote the pointwise stabiliser of $x_1,\dots,x_n\in X$.  
\end{itemize}

\subsection{Ample generics}

Let $G$ be a topological group.
We denote the conjugation action of $G$ on itself by $h^g:= g^{-1}hg$ for $g,h\in G$.
For $n\in\N$ this action naturally extends to the \emph{diagonal conjugation action} of $G$ on $G^n$ by
$\hh^g:=(h_1^g,\dots,h_n^g)$ where $\hh=(h_1,\dots,h_n)\in G^n$.
We say that $\hh\in G^n$ is a \emph{generic} $n$-tuple 
if its orbit $\hh^G$ under this action is comeagre in $G^n$,
i.e.~if it contains the intersection of countably many dense open subsets of $G^n$.
The group $G$ is said to have \emph{ample generics} if it has generic $n$-tuples for all $n\in\N$.

Having ample generics is a strong condition connecting algebraic and topological features of $G$.
It implies several other properties (see Section \ref{sec:cag} below) and has thus been an object of study over
the years.
It was introduced by Kechris and Rosendal \cite{KR:TAG} and many prominent permutation groups
have been shown to have this property:
\begin{itemize}
\item the symmetric group of countably infinite degree by  Kechris and Rosendal \cite{KR:TAG},
\item the automorphism group of the random graph by Hrushovski \cite{Hr:EPI},
\item the group of isometries of the rational Urysohn space by Solecki \cite{So:EPI},
\item the group of homeomorphisms $\Homeo 2^\omega$ of the Cantor space $2^\omega$ by  Kwiat\-kow\-ska \cite{Kw:GHC}.
\end{itemize}  
Conversely,
Hodkinson showed that the automorphism group of $(\Q,<)$ does not have generic pairs~\cite{KT:GAU}.

 It is worth noting that the groups that have thus far been shown to have ample generics tend to be automorphism groups of countable homogenous relational structures, homeomorphism groups of topological spaces or isometry groups of metric spaces. There is a notable absence of examples of automorphism groups of algebraic structures. One exception is the above-mentioned result of Kwiatkowska~\cite{Kw:GHC},
 because the homeomorphism group of the Cantor space can be viewed as the automorphism group of the countable atomless Boolean algebra.
 As a crucial first step towards our Theorem~\ref{thm:DAG} in the non-abelian case, we will generalize her result as
 follows.

\begin{thm} \label{thm:H}
 The pointwise stabiliser of finitely (possibly zero) many points in the homeomorphism group of the Cantor space
 $2^\omega$  has ample generics.
\end{thm}

Note that the isomorphism type of the stabiliser above depends only on the number of fixed points
and not on their specific choice.

For zero fixed points we recover Kwiatkowska's result for $\Homeo 2^\omega$ \cite{Kw:GHC}.
Our proof of Theorem~\ref{thm:H} in Section~\ref{sec:X0} broadly follows  Kwiat\-kow\-ska's and also
makes use of some of her findings along the way. 

 Finally we state some straightforward consequences of our proof of Theorem~\ref{thm:DAG} in the abelian case,
 which may be folklore but we have not been able to find in the literature.

\begin{cor} \label{cor:GL}
 For the countable-dimensional vector space $\V$ over a finite or countably infinite field,
 the general linear group $\GL \V$, the projective linear group $\PGL\V$ and the affine general linear group
 $\mathrm{AGL}\, \V$ have ample generics. 
\end{cor} 

 This in particular gives a new unified approach to Evans' result that $\GL\V$ has the small index property~\cite{Ev:SSI}
 and Tolstykh's result that $\GL\V$ has the Bergman property~\cite{To:IGL}.





\subsection{Mal'cev algebras and Boolean powers} \label{sec:algebra}
We refer to~\cite{BS:CUA,MMT:ALV1} for background on general algebra.
 Let $\sigma$ be a signature consisting of function symbols $f_j$ with arity $n_j \geq 0$ for $j \in J$.
 An \emph{algebra} $\AA = (A; (f_j^\AA)_{j\in J})$ of signature $\sigma$ is a set $A$ (the universe) together with functions
 $f_j^\AA\colon A^{n_j}\to A$ for all $j \in J$.
 Throughout, we will use bold letters to denote algebras and matching standard letters for the underlying universes.

 An algebra with more than one element is \emph{simple} if its only congruences are the total relation and the equality.
 This is consistent with the classical notion for groups and rings.

 An algebra $\AA$ is \emph{Mal'cev} if it has a ternary term operation $m$ satisfying $m(x,x,y) = y = m(y,x,x)$ for all
 $x,y\in A$.
 Note that every group $(G;\cdot,^{-1},1)$ has a Mal'cev term $m(x,y,z) = xy^{-1}z$.
 Similarly quasigroups, loops, rings as well as their expansions with additional operations are all Mal'cev.

 The generalization of commutator theory from groups to algebras as described in~\cite{FM:CTC} allows us to talk
 about (non-)abelian algebras. A Mal'cev algebra is \emph{abelian} if all its operations are affine functions over
 some module; otherwise we call it \emph{non-abelian}. Clearly every module is abelian.
 Abelianness as above is consistent with the classical notion for groups.
 A loop is abelian if and only if it is an abelian group.
 A ring is abelian if and only if its multiplication is constant $0$ (see~\cite{FM:CTC}).
 
 We refer to \cite{Ho:MT} for basics on Fra{\"i}ss{\'e} theory.
 As we discuss in~\cite{MR:FBP}, for $\AA$ a finite simple non-abelian Mal'cev algebra, the set of finite direct
 powers $K := \{ \AA^n\st n\in\N \}$
 has the Joint Embedding Property and the Amalgamation Property but in general not the Hereditary Property.
 Hence $K$ has a (generalized) Fra{\"i}ss{\'e} limit.
In \cite{MR:FBP} this Fra{\"i}ss{\'e} limit
is explicitly described as a filtered Boolean power of $\AA$,
a generalization of direct powers introduced by Arens and Kaplansky~\cite{AK:TRA} (see also~\cite{Ev:EAC}).
 For $\AA$ abelian, we establish the corresponding results in Section~\ref{sec:abelian}.
 Over time Boolean powers have proved a useful tool for representing algebras and transferring properties of Boolean
 algebras to arbitrary structures~\cite{Bu:BP,Fo:GBT}.
 In groups they were studied for their interesting combinatorial properties~\cite{Ap:BPG,BE:SIP}.
 We briefly review this construction.

For a topological space $X$ and a finite algebra $\AA$ endowed with the discrete topology, a function $f\colon X\to A$
is \emph{continuous} if for every $a\in A$ the pre-image $f^{-1}(a)$ is clopen in $X$. The set of all continuous functions
from $X$ to $\AA$, 
\[ \{ f\colon X\to A \st f \text{ is continuous} \}, \]
 is the universe of an algebra under the pointwise action of the operations from $\AA$.
 We denote this \emph{algebra of continuous functions}
 $\AA^X$. 
 When $X$ is also equipped with the discrete topology, $\AA^X$ is the direct power of $\AA$ consisting of
 \emph{all} functions from $X$ to $\AA$. Throughout the paper it will always be clear from context what the
 topology on $X$ is.

 For a (possibly empty) tuple of singleton subalgebras  $\be := (e_1,\dots,e_n)$ of $\AA$,
 and a tuple of distinct points $\bx := (x_1,\dots,x_n)$ in $X$,
 we denote the subalgebra of $\AA^X$ with universe
\[ \{ f\in \AA^X \colon f(x_1)=e_1,\dots,f(x_n)=e_n \} \]
 as $(\AA^X)^\bx_\be$.

 The set of clopens in a topological space forms a Boolean algebra. The above construction is of particular significance when $X$ is a Stone space, i.e. a compact Hausdorff totally disconnected space.
 These are precisely the spaces corresponding to Boolean algebras via Stone duality. In that case
 we call $\AA^X$ a \emph{Boolean power} and  $(\AA^X)^\bx_\be$ a \emph{filtered Boolean power} of $\AA$.
 For basics on Boolean algebras we refer the reader to \cite[Chapter 1]{Mo:HBA1} and to 
 \cite[Chapter~3]{Mo:HBA1} for Stone duality.
 We remark that in \cite{MR:FBP} the Boolean powers $\AA^X$ and $(\AA^X)^\bx_\be$ were denoted $\AA^\BB$ and $(\AA^\BB)^\bx_\be$,
 respectively, where $\BB$ is the Boolean algebra corresponding to the Stone space $X$.

 A Stone space of particular importance in this paper is the Cantor space $2^\omega$.
 The Boolean algebra of its clopens is the unique (up to isomorphism) countable atomless Boolean algebra.
 This is in fact the Fra{\"i}ss{\'e} limit of all finite Boolean algebras.

\subsection{Consequences of ample generics} \label{sec:cag}

In \cite{KR:TAG} Kechris and Rosendal proved that if $G$ is a Polish group (i.e.~if it is separable and completely metrisable)
with ample generics, then $G$ also has the following (see \cite{Mac:SHS} for background):
\begin{itemize}
\item
 \emph{small index property}, i.e.~every subgroup of $G$ of index less than $2^{\aleph_0}$ is
  open~\cite[Theorem 1.6]{KR:TAG},
\item \emph{automatic continuity}, i.e.~every homomorphism $\pi\colon G\to H$ where $H$ is
  a separable topological group is continuous \cite[Theorem 6.24]{KR:TAG},
\item  
 a unique Polish group topology~\cite[Corollary 6.25]{KR:TAG}.
\end{itemize}
Moreover, if $G$ is the automorphism group of an $\omega$-categorical structure and has ample generics,
then $G$ has uncountable strong cofinality.
The \emph{strong cofinality} of $G$ is the smallest cardinality $\kappa$ such that $G$ is the union
of a chain $(U_i)_{i<\kappa}$ of proper subsets of $G$ such that for all $i<j$ we have $U_i\subseteq U_j$, 
 $U_i = U_i^{-1}$ and $U_iU_i \subseteq U_k$ for some $k<\kappa$. 

 By~\cite[Theorem 2.2]{DH:GAG} a group $G$ has uncountable strong cofinality if and only if $G$ has
\begin{itemize}
\item
 \emph{uncountable cofinality},
 i.e.~$G$ is not the union of a countable chain of proper subgroups, and
\item
 the \emph{Bergman property}, i.e.~for every generating set $X$ of $G$ there exists $n\in\N$ such
 that every element of $G$ is a product of $\leq n$ elements of $X$ and their inverses.
\end{itemize}

In~\cite[Theorem 3.9]{MR:FBP} we proved that the filtered Boolean power $\ABxe$ is $\omega$-categorical.
Therefore Theorem~\ref{thm:DAG} implies the following.

\begin{cor}
  \label{cor:DAG}
 With the notation and assumptions of Theorem~\ref{thm:DAG},
 the automorphism group of the filtered Boolean power $\ABxe$ has the small index property,
 automatic continuity, unique Polish group topology, 
 uncountable (strong) cofinality and the Bergman property.
\end{cor}

 Note that points $x_1,\dots,x_n\in 2^\omega$ can be identified with maximal ideals (or ultrafilters) of the countable
 atomless Boolean algebra $\BB$. Then the augmented Boolean algebra $(\BB;x_1,\dots,x_n)$ is a structure
 with operations from the Boolean algebra and unary relations that are ideals $x_1,\dots,x_n$
 (see \cite[p. 137]{MR:CRN}).
 As explained in the proof of \cite[Theorem 3.9]{MR:FBP} it follows that $(\BB;x_1,\dots,x_n)$ is
 $\omega$-categorical by~\cite[Theorem~7]{MR:CRN}.
 Since $\Aut (\BB;x_1,\dots,x_n)$ is isomorphic to $(\Homeo 2^\omega)_{(x_1,\dots,x_n)}$,
 Theorem~\ref{thm:H} implies the following.

\begin{cor} \label{cor:H}
 The pointwise stabiliser of finitely (possibly zero) many points in the homeomorphism group of $2^\omega$
 has the small index property,
 automatic continuity, unique Polish group topology, 
 uncountable (strong) cofinality and the Bergman property.
\end{cor}

In Corollaries~\ref{cor:DAG}, \ref{cor:H}, automatic continuity as defined above is a
new consequence of Theorems~\ref{thm:DAG}, \ref{thm:H}, respectively.
The other properties listed in these corollaries have already been established by different means in~\cite{MR:FBP}.

\section{The abelian case}  \label{sec:abelian}

In this section we will prove Theorem~\ref{thm:DAG} for $\AA$ a finite simple abelian Mal'cev algebra.
Such $\AA$ are characterized as specific reducts of finite simple modules (see Section~\ref{sec:ab} below).
We first see that the class $K$ of finite direct powers of $\AA$ is Fra{\"i}ss{\'e} with limit the Boolean power $\AB$.
In fact, we show that structures in $K$ expanded with $n$ partial isomorphisms have the JEP and the
weak amalgamation property (WAP). By Kechris and Rosendal~\cite[Theorem 6.2]{KR:TAG}
this is equivalent to the  Fra{\"i}ss{\'e} limit of $K$ having generic $n$-tuples.
Finally, since all countably infinite subalgebras of a direct power of $\AA$ turn out to be isomorphic to $\AB$,
this will prove Theorem~\ref{thm:DAG} for $\AA$ abelian.

\subsection{A characterization of ample generics}
 Let $K$ be a Fra{\"i}sse class and $n\geq 0$.
 Following Kechris and Rosendal~\cite{KR:TAG}, an \emph{$n$-system} $S := (\CC,\psi_1,\dots,\psi_n)$
 consists of $\CC\in K$ with $n$ \emph{partial isomorphisms}, i.e. isomorphisms
 $\psi_1\colon\DD_1\to\E_1,\dots,\psi_n\colon\DD_n\to\E_n$ between substructures of $\CC$.
 An \emph{emdedding} $f$ of one $n$-system
 $S := (\CC,\psi_1\colon\DD_1\to\E_1,\dots,\psi_n\colon\DD_n\to\E_n)$ into another $n$-system 
 $S' := (\CC',\psi_1'\colon\DD'_1\to\E'_1,\dots,\psi'_n\colon\DD'_n\to\E'_n)$ is a function $f\colon \CC\to\CC'$
 embedding $\CC$ into $\CC'$, $\DD_i$ into $\DD_i'$ and $\E_i$ into $\E_i'$ such that $f\psi_i \subseteq \psi'_if$
 for all $i\in [n]$. We note that this condition is weaker than the standard model theoretic definition of an embedding
 when $S$ and $T$ are considered as structures.

 Let $K_p^n$ be the class of all $n$-systems for $\CC\in K$ as defined above.
 Then $K_p^n$ satisfies the \emph{weak amalgamation property (WAP)} if for all $S\in K_p^n$ there exist $T\in K_p^n$ and an
embedding $e\colon S\to T$ such that for all $T_1,T_2\in K_p^n$ and embeddings $f_1\colon T\to T_1, f_2\colon T\to T_2$
there exist $U\in K_P^n$ and embeddings $g_1\colon T_1\to U, g_2\colon T_2\to U$ such that $g_1f_1e = g_2f_2e$.
\begin{center}
  \begin{tikzcd}[row sep=small]
 &   & T_1 \arrow[dotted]{rd}{g_1} & \\
S \arrow[dotted]{r}{e} & T \arrow[ru]{}{f_1} \arrow[rd]{}[swap]{f_2} & & U \\
& & T_2 \arrow[dotted]{ru}[swap]{g_2} &
\end{tikzcd}
\end{center}

The class $K_p^n$ satisfies the \emph{cofinal amalgamation property (CAP)} if it has a cofinal (under embeddings)
subclass that has the AP. Clearly the CAP implies the WAP.

 We will use this observation and the following characterization of the existence of ample generics.

\begin{thm} \cite[Theorem 6.2]{KR:TAG} \label{thm:WAP}
 Let $K$ be a Fra{\"i}sse class. 
 Then the automorphism group of the Fra{\"i}sse limit of $K$ has ample generics if and only if $K_p^n$ has the JEP
 and the WAP for all $n\in\N$.
\end{thm}

\subsection{Finite simple abelian Mal'cev algebras} \label{sec:ab}

 Two algebras on the same universe are \emph{term equivalent} if they both have the same term functions; 
 they are \emph{polynomially equivalent} if they both have the same polynomial functions, i.e.,
 functions composed from the basic operations as well as constants on the universe.
 For example, the group $(\Z_2,+)$ and its idempotent reduct $(\Z_2,x+y+z)$ are polynomially but not term equivalent.
 Note that polynomially equivalent algebras have the same congruences, while term equivalent algebras additionally
 have the same subalgebras and homomorphisms. 
 
 Let $\AA$ be a finite simple abelian Mal'cev algebra.
 Then $\AA$ is polynomially equivalent to a finite simple module isomorphic to
\begin{equation} \label{eq:module}
 \MM := (F^m,+,(ax)_{a\in F^{m\times m}})
\end{equation}
 for some finite field $\FF = (F,+,\cdot)$ and $m\in\N$.
 More precisely, by~\cite[Corollary 2.12]{Sz:CUA} there exist 
 $0\leq r\leq m$ and $p_r := \diag(\underbrace{1,\dots,1}_{r},0\dots,0)\in F^{m\times m}$ such that
 $\AA$ is term equivalent to an algebra isomorphic to one of the following reducts of $\MM$:
\begin{equation} \label{eq:ab1}
 \AA_1 := \bigl(F^m, x-y+z, \bigl( ax+(1-a)y) )_{a\in F^{m\times m}}, p_rx\bigr)
\end{equation}
 or
\begin{equation} \label{eq:ab2}
 \AA_2 := \bigl(F^m, x-y+z, \bigl( ax+(1-a)y) )_{a\in F^{m\times m}}, p_rx, (x+c)_{c\in F^m}\bigr).
\end{equation}
Note that the proper subalgebras of $\AA_1$ are exactly the $1$-element subsets of $p_r(M) = F^r\times\{0\}^{m-r}$,
whereas $\AA_2$ has no proper subalgebras.

Using this characterization, the description of subalgebras and (partial) homomorphisms of
direct powers of $\AA$ is straightforward and does does not require the finiteness assumption on $\FF$.
For $k\in\N$ the elements of $\MM^k$ and also of $\AA_i^k$ for $i\in [2]$
can naturally be identified with $k\times m$ matrices over $F$.
 Specifically, for $x_1,\dots,x_k\in F^m$, we write the general element of  $\MM^k$  as the matrix
 \[ \bx := \begin{bmatrix} x_1 \\ \vdots \\ x_k \end{bmatrix} \in F^{k\times m}.\]
 A matrix $P = (p_{ij})\in F^{\ell\times k}$ for $k,\ell\in\N$ then naturally induces a module homomorphism  
 $\MM^k\rightarrow \MM^\ell$ via
 \[ \bx\mapsto P\bx :=  \begin{bmatrix} \sum_{j=1}^k p_{1j} x_j \\ \vdots \\ \sum_{j=1}^k p_{\ell j} x_j \end{bmatrix} \in M^\ell. \]
 
\begin{lemma}  \label{lem:Hrushovski}
  Let $k,\ell\in\N$. For
$\MM, \AA_1,\AA_2$ over an arbitrary field $\FF$ as in~\eqref{eq:module},
\eqref{eq:ab1}, \eqref{eq:ab2}, respectively,
 we have:
\begin{enumerate}
\item \label{it:ab1}
 $U\leq \AA_1^k$ if and only if $U = \bd+V$ for some $\bd\in p_r(M)^k$ and $V\leq\MM^k$.
\item \label{it:ab2}
  $U\leq \AA_2^k$ if and only if $U = \bd+V$ for some $\bd\in p_r(M)^k$ and $\{(c,\dots,c) \st c\in M\} \leq V\leq\MM^k$.
\item  \label{it:ab1hom}
 $\varphi\colon\AA_1^k\to\AA_1^\ell$ is a homomorphism if and only if there exist $P\in F^{\ell\times k}$ and
 $\bd\in p_r(M)^\ell$ such that
 \[ \varphi(\bx) = P\bx+\bd \text{ for all } \bx \in M^k. \]
\item \label{it:ab2hom}
 $\varphi\colon\AA_2^k\to\AA_2^\ell$ is a homomorphism if and only if there exist $P = (p_{ij})\in F^{\ell\times k}$
 with $\sum_{j=1}^k p_{ij} = 1$ for all $i\in [\ell]$ and $\bd\in p_r(M)^\ell$ such that
 \[ \varphi(\bx) = P\bx+\bd \text{ for all } \bx \in M^k. \]
\item  \label{it:ab}
  (HP)
  For $i\in [2]$,  every subalgebra $U\leq \AA_i^k$ is isomorphic to $\AA_i^n$ for some $0\leq n\leq k$. 
\item \label{it:Hrushovski}
 (cf. EPPA, Hrushovski property~\cite[Def. 6.3]{KR:TAG})
 For $i\in [2]$ and $U\leq \AA_i^k$, every embedding $\varphi\colon U\to\AA_i^k$ extends to an automorphism of $\AA_i^k$.
\end{enumerate}    
\end{lemma}

\begin{proof}
 We give brief sketches of the standard Linear Algebra proofs.
  
\eqref{it:ab1}
 Since any $U \leq \AA_1^k$ is closed under the Mal'cev term $x-y+z$, it is a coset $\bd+V$ of a submodule
 $V\leq\MM^k$. Closure under $p_r$ yields $\bd\in  p_r(M)^k$.
 Conversely, every such coset is closed under all operations of $\AA_1$, hence a subalgebra of $\AA_1^k$.

\eqref{it:ab2} is as the previous case with the additional condition that $U$ is closed under translations
 $x+c$ for $c\in M$, which forces $V$ to contain the diagonal $\{(c,\dots,c) \st c\in M\}$.

\eqref{it:ab1hom}
 Since any homomorphism $\varphi$ preserves the Mal'cev term $x-y+z$, it is an affine map.  
 Since any endomorphism on $\AA_1$ preserves $ax+(1-a)y$ for all $a\in F^{m\times m}$ and $p_r$,
 it turns out to be of the form $x\mapsto fx+d$ for some $f\in F$ and $d\in p_r(F^m)$.
 If follows that $\varphi(\bx) = P\bx+\bd$ as required.
The converse is clear.

\eqref{it:ab2hom} is as the previous case with the additional condition that $\varphi$ preserves the
translations $x+c$, which forces $P$ to send the diagonal vector $(c,\dots,c)\in F^k$ to $(c,\dots,c)\in F^\ell$
for all $c\in M$, i.e. the sum of every row of $P$ is $1$.

\eqref{it:ab}
 From the previous items it follows that every subalgebra of $\AA_i^k$ occurs as isomorphic image of a power
 $\AA_i^n$ for $0\leq n\leq k$.

\eqref{it:Hrushovski} follows using the previous items.
\end{proof}

\subsection{Amalgamation}
Fix $i\in [2]$, and let $K$ denote the class of isomorphic copies of finite direct
powers of~$\AA_i$ over an arbitrary field $\FF$ as in~\eqref{eq:ab1}, \eqref{eq:ab2}
(including the singleton algebra $\AA_i^0$ if $i=1$).

Recall that $\AA_i$ is polynomially equivalent to a module $\MM$ with zero~$0$ as in~\eqref{eq:module}. 
For $k\in\N$, note that
\[ \Aut\AA_i^k\cap\Aut\MM^k = \{ \phi\in\Aut \AA_i^k \st \phi(0,\dots,0) = (0,\dots,0) \} \]
by Lemma~\ref{lem:Hrushovski}\eqref{it:ab1hom}\eqref{it:ab2hom}.

 For $n\geq 0$, recall that $K_p^n$ is the class of algebras in $K$ expanded with $n$ partial isomorphisms.
 Let $K_0^n \subseteq K_p^n$ be the subclass of $n$-systems with $n$ \emph{total module automorphisms},
 i.e. isomorphic copies of $(\AA_i^k,\phi_1,\dots,\phi_n)$ for
$\phi_1,\dots,\phi_n\in\Aut\AA_i^k\cap\Aut\MM^k$.
 We will show that $K_0^n$ is cofinal in $K_p^n$ and has the HP, JEP and AP.
 Since $K = K_0^0$, this also yields that $K$ itself is a Fra{\"i}ss\'e class.
                                             
\begin{lemma}   \label{lem:cofinal}
 $K_0^n$ is cofinal in $K_p^n$.
\end{lemma}
  
\begin{proof}
 By Lemma~\ref{lem:Hrushovski}\eqref{it:Hrushovski}, the class of expansions of powers of $\AA_1, \AA_2$,
 respectively, with $n$ total automorphisms is cofinal in $K_p^n$.
 It remains to show that these automorphisms can be chosen to fix the zero vector $\bo$ in these powers.

 For this let $k\in\N$ and $\psi_1,\dots,\psi_n\in\Aut\AA_i^k$. Let $\be := (e_1,\dots,e_r)$ be a basis for
 the $F$-vector space $F^r\times \{0\}^{m-r}$. Let $\bo := (0,\dots,0)\in M^k$.
 Since $\psi_p(\bo) \in (F^r\times\{0\}^{m-r})^k$
 by Lemma~\ref{lem:Hrushovski}\eqref{it:ab1hom}\eqref{it:ab2hom},
 we have $c^{p}_{ij}\in F$ such that 
\[ \psi_p(\bo) = (\sum_{j=1}^r c^p_{ij}e_j)_{i\in [k]}. \]

 {\bf Case $\AA_1$ as in~\eqref{eq:ab1}:}
 Identifying the concatenation of $\bx\in M^k$ and $\be\in M^r$ as
 $(k+r)$-tuple over $M$,
\[ h\colon \AA_1^k\to\AA_1^{k+r},\ \bx \mapsto (\bx,\be), \] 
is an embedding since $e_1,\dots,e_r$ each form a one-element subalgebra of~$\AA_1$. 
 For $p\in [n]$ define
\[ \phi_p \colon \AA_1^{k+r}\to\AA_1^{k+r},\ (\bx,\by) \mapsto \bigl(\psi_p(\bx)-\psi_p(\bo)+(\sum_{j=1}^r c^p_{ij}y_j)_{i\in [k]},\, \by), \]
where $\bx\in M^k$ and $\by := (y_1,\dots,y_r)\in M^r$.
Then $\phi_p$ is an automorphism on $\AA_1^{k+r}$ fixing the zero vector. Further
\[ h\psi_p(\bx) = (\psi_p(\bx),\be) = \phi_p h(\bx) \quad \text{for all } \bx\in M^k. \]
Hence $h$ embeds $(\AA_1^k,\psi_1,\dots,\psi_n)$ into $(\AA_1^{k+r},\phi_1,\dots,\phi_n)$, which proves the lemma for
$\AA_1$ as in~\eqref{eq:ab1}.

{\bf Case $\AA_2$ as in~\eqref{eq:ab2}:}
Recall that the translation by $e_i$ for $i\in [r]$ is a basic operation of $\AA_2$,
and note that it is also an automorphism by Lemma~\ref{lem:Hrushovski}\eqref{it:ab2hom}.
 For $\bx\in M^k$, we denote the corresponding componentwise translation $\bx+e_i \in M^k$.
 Identifying the concatenation of $\bx, \bx+e_1,\dots,\bx+e_r\in M^k$ as $k(r+1)$-tuple over $M$,
\[ h\colon \AA_2^k\to\AA_2^{k(r+1)},\ \bx \mapsto (\bx,\bx+e_1,\dots,\bx+e_r), \] 
is an embedding.

Let $\pi_1\colon \AA_2^k\to \AA_2,\ (x_1,\dots,x_k)\mapsto x_1,$ be the projection on the first component.
For $p\in [n]$ and $\by_0,\dots,\by_r\in M^k$, define
\[ \gamma_p\colon \AA_2^{k(r+1)}\to\AA_2^{k},\ (\by_0,\dots,\by_r) \mapsto (\sum_{j=1}^r c^p_{ij} \pi_1(\by_j-\by_0))_{i\in [k]}. \] 
Then $\gamma_ph(\bx) =  (\sum_{j=1}^r c^p_{ij}e_j)_{i\in [k]} = \psi_p(\bo)$ for all $\bx\in M^k$.
 Note that with respect to the standard basis $\gamma_p$ has a matrix representation $[C_0\, C_1\dots\, C_r] \in F^{k\times k(r+1)}$ for $k\times k$ blocks $C_i$ such that $\sum_{i=1}^r C_i = -C_0$.

 Next define
\begin{align*}
  \phi_p\colon \AA_2^{k(r+1)} &\to\AA_2^{k(r+1)},\\
  (\by_0,\dots,\by_r) &\mapsto \bigl(\psi_p(\by_0)-\psi(\bo)+\gamma_p(\by_0,\dots,\by_r),\dots,\psi_p(\by_r)-\psi(\bo)+\gamma_p(\by_0,\dots,\by_r)\bigr). \end{align*}
 Since $\gamma_p$ fixes the zero vector by its definition, so does $\phi_p$.
 To see that $\phi_p$ is an automorphism of $\AA_2^{k(r+1)}$, first let $P\in F^{k\times k}$ be the matrix representation of
 the linear map $\AA_2^k\to\AA_2^k, \bx\mapsto\psi_p(\bx)-\psi_p(\bo)$,
 for the standard basis. Then $\phi_p$ has a representation by the block matrix
 \[ \begin{bmatrix} P & 0 &\dots & 0 \\
      0 & P & &\vdots \\
      \vdots && \ddots & 0 \\
      0 & \dots & 0 &P \end{bmatrix} + \begin{bmatrix} C_0 & C_1 & \dots & C_r \\ C_0 & C_1 & \dots & C_r \\ \vdots & \vdots & &\vdots \\ C_0 & C_1 & \dots & C_r \end{bmatrix}. \]
  To see that this matrix is regular, first add block columns $2$ to $r+1$ to the first block column to get
 \[ \begin{bmatrix} P & 0 &\dots & 0 \\
      P & P & &\vdots \\
      \vdots && \ddots & 0 \\
      P & \dots & 0 &P \end{bmatrix} + \begin{bmatrix} 0 & C_1 & \dots & C_r \\ 0 & C_1 & \dots & C_r \\ \vdots & \vdots & &\vdots \\ 0 & C_1 & \dots & C_r \end{bmatrix}. \]
 Then subtract the first block row from each of the following to obtain
\[ \begin{bmatrix} P & 0 &\dots & 0 \\
      0 & P & &\vdots \\
      \vdots && \ddots & 0 \\
      0 & \dots & 0 &P \end{bmatrix} + \begin{bmatrix} 0 & C_1 & \dots & C_r \\ 0 & 0 & \dots & 0 \\ \vdots & \vdots & &\vdots \\ 0 & 0 & \dots & 0 \end{bmatrix}, \]
  which is regular since $P$ is. It then follows that
  the original matrix representation of $\phi_p$ is regular and $\phi_p$ is an automorphism.

 Finally
\[ h\psi_p(\bx) = (\psi_p(\bx),\psi_p(\bx)+e_1,\dots,\psi_p(\bx)+e_r) = \phi_p h(\bx) \quad \text{for all } \bx\in M^k. \]
Hence $h$ embeds $(\AA_2^k,\psi_1,\dots,\psi_n)$ into $(\AA_2^{k(r+1)},\phi_1,\dots,\phi_n)$, which proves the lemma for $\AA_2$
as in~\eqref{eq:ab2}.
\end{proof}

Although an embedding $f\colon S\to T$ for $S,T\in K_0^n$ may not map the zero in $S$ to the zero in $T$,
it factors an embedding which does that.

\begin{lemma} \label{lem:f0}
 For $i\in [2]$, let $S := (\AA_i^k,\psi_1,\dots,\psi_n)$, $T:= (\AA_i^{\ell},\phi_1,\dots,\phi_n)\in K_0^n$
 with an embedding $f\colon S\to T$. Then
 \[ g\colon \AA_i^\ell\to\AA_i^\ell,\ \bx\mapsto \bx-f(\bo_{M^k}), \]
 is an automorphism of $T$ and $gf(\bo_{M^k}) = \bo_{M^\ell}$.
\end{lemma}

\begin{proof}
  Clearly $g\in\Aut\AA_i^\ell$ and $gf(\bo_{M^k}) = \bo_{M^\ell}$.
  To see that $g$ preserves $\phi_p$ for $p\in [n]$, let $\bx\in M^\ell$. Since $\phi_p\in\Aut\MM^\ell$, we have
 \[ \phi_p g(\bx) = \phi_p(\bx-f(\bo)) = \phi_p(\bx) - \underbrace{\phi_p f(\bo)}_{= f\psi_p(\bo)}  = \phi_p(\bx)-f(\bo) = g\phi_p(\bx). \]
\end{proof}

 Finally we are ready to show that $K_0^n$ is a Fra{\"i}ss{\'e} class.

\begin{lemma}   \label{lem:K0n}
 $K_0^n$ has the HP, JEP and AP.
\end{lemma}
  
\begin{proof}
 The HP is immediate from Lemma~\ref{lem:Hrushovski}\eqref{it:ab}.
 The JEP will follow from the AP since every structure in $K_0^n$ embeds $(0,\id,\dots,\id)$
 if $i=1$ or $(\AA_2,\id,\dots,\id)$ if $i=2$.
 For the AP, let $k,\ell_1,\ell_2\in\N$, and $S:=(\AA_i^k,\sigma_1,\dots,\sigma_n)$,
 $T_1:=(\AA_i^{\ell_1},\psi_1,\dots,\psi_n)$, $T_2:= (\AA_i^{\ell_2},\phi_1,\dots,\phi_n)\in K_0^n$ with embeddings
 \[ f_1\colon S\to T_1, \quad f_2\colon S\to T_2. \]
 By Lemma~\ref{lem:f0} we may assume that $f_j(\bo_{M^k})=\bo_{M^{\ell_j}}$ for $j\in [2]$. 
 We construct an amalgam as a quotient of $T_1\times T_2$ by considering the two cases $\AA_1$ and $\AA_2$.
 
{\bf Case $\AA_1$ as in~\eqref{eq:ab1}:} Let
\[ E := \{ (f_1(\bx),-f_2(\bx)) \st \bx\in M^k \} \leq \MM^{\ell_1}\times\MM^{\ell_2}. \]
 Let $\equiv$ be the congruence on $\AA_1^{\ell_1}\times\AA_1^{\ell_2}$ defined by
\[ (\by_1,\bz_1) \equiv (\by_2,\bz_2) \text{ if } (\by_1,\bz_1) - (\by_2,\bz_2) \in E. \]
 That is, the congruence classes of $\equiv$ are the cosets of $E$ in $\MM^{\ell_1}\times\MM^{\ell_2}$.
 Let $D\leq \MM^{\ell_1}\times\MM^{\ell_2}$ be a direct module complement of $E$. Then $D$ is a transversal through
 the cosets of $E$ and $D$ is a subalgebra of $\AA_1^{\ell_1}\times\AA_1^{\ell_2}$ by
 Lemma~\ref{lem:Hrushovski}\eqref{it:ab1}.
 Hence the quotient $\R := (\AA_1^{\ell_1}\times\AA_1^{\ell_2})/\equiv$ is isomorphic to $\AA_1^{\ell_1+\ell_2-k}$.

 Further we claim that $(\psi_p,\phi_p)$ preserves $\equiv$ for all $p\in [n]$.
 To see that let $(\by_1,\bz_1) \equiv (\by_2,\bz_2)$, that is,
\[ (\by_1,\bz_1) - (\by_2,\bz_2) = (f_1(\bx),-f_2(\bx)) \text{ for some } \bx\in M^k. \] 
Applying $(\psi_p,\phi_p)$ componentwise on both sides and using that they are module homomorphisms yields
\begin{align*}
  (\psi_p(\by_1),\phi_p(\bz_1)) - (\psi_p(\by_2),\phi_p(\bz_2))
  & = ( \psi_p f_1(\bx),\ -\phi_pf_2(\bx) ) \\
  & = ( f_1\sigma_p(\bx)  , -f_2\sigma_p(\bx)) \in E. 
\end{align*}
 Hence $\equiv$ is a congruence on $T_1\times T_2$ with quotient $R := (T_1\times T_2)/\equiv$ in $K_0^n$.

 It is now straightforward that
\begin{align*}
 & h_1\colon T_1\to R,\ \by\mapsto (\by,\bo)+E, \\
 & h_2\colon T_2\to R,\ \bz\mapsto (\bo,\bz)+E,
\end{align*}
 are embeddings and $h_1f_1=h_2f_2$. Hence $K_0^n$ has the AP for $\AA_1$ as in~\eqref{eq:ab1}. 

 {\bf Case $\AA_2$ as in~\eqref{eq:ab2}:} We adapt the arguments from the previous case.
 Let $U\leq\MM^k$ be a direct complement of the diagonal $D := \{(c,\dots,c) \st c\in M \}\leq\MM^k$.
 Then $\AA_2^k/U \cong\AA_2$. Let
\[ E := \{ (f_1(\bx),-f_2(\bx)) \st \bx\in U \} \leq \MM^{\ell_1}\times\MM^{\ell_2} \]
 and let $\equiv$ be the congruence on $\AA_2^{\ell_1}\times\AA_2^{\ell_2}$ defined by
\[ (\by_1,\bz_1) \equiv (\by_2,\bz_2) \text{ if } (\by_1,\bz_1) - (\by_2,\bz_2) \in E. \]
 Again the congruence classes of $\equiv$ are the cosets of $E$ in $\MM^{\ell_1}\times\MM^{\ell_2}$.
 Since $f_1,f_2$ are embeddings, the only diagonal element in $E$ is $(\bo_{M^{\ell_1}},\bo_{M^{\ell_2}})$.
 By Lemma~\ref{lem:Hrushovski}\eqref{it:ab2} there exists a subalgebra of $\AA_2^{\ell_1}\times\AA_2^{\ell_2}$
 which is a transversal through the cosets of $E$.
 Hence the quotient $\R := (\AA_2^{\ell_1}\times\AA_2^{\ell_2})/\equiv$ is isomorphic to $\AA_2^{\ell_1+\ell_2-k+1}$.

 Further we claim that $(\psi_p,\phi_p)$ preserves $\equiv$ for all $p\in [m]$.
 To see that let $(\by_1,\bz_1) \equiv (\by_2,\bz_2)$, that is,
\[ (\by_1,\bz_1) - (\by_2,\bz_2) = (f_1(\bx),-f_2(\bx)) \text{ for some } \bx\in U. \] 
 Applying $(\psi_p,\phi_p)$ componentwise on both sides and using that they are group homomorphisms yields
\begin{align*}
  (\psi_p(\by_1),\phi_p(\bz_1)) - (\psi_p(\by_2),\phi_p(\bz_2))
  & = ( \psi_p f_1(\bx),\ -\phi_pf_2(\bx) ) \\
  & = ( f_1\sigma_p(\bx)  , -f_2\sigma_p(\bx)) \in E. 
\end{align*}
 Hence $\equiv$ is a congruence on $T_1\times T_2$ with quotient $R := (T_1\times T_2)/\equiv$ in $K_0^n$.

 To define embeddings from $T_j$ into $R$ for $j\in [2]$,
 let $f_j(U)\leq W_j\leq\MM^{\ell_j}$ be a direct complement to the diagonal $\{(c,\dots,c) \st c\in M \}$ in $\MM^{\ell_j}$.
 Then every element in $M^{\ell_j}$ can be written uniquely as a translate $\by+c$ for $\by\in W_j$ and $c\in M$.
 Define 
\begin{align*}
 & h_1\colon T_1\to R\ \text{ by } h_1(\by+c) := (\by,\bo_{M^{\ell_2}})+c+E \text{ for } \by\in W_1,c\in M, \\
 & h_2\colon T_2\to R\ \text{ by } h_2(\bz+c) := (\bo_{M^{\ell_1}},\bz)+c+E \text{ for } \bz\in W_2,c\in M.
\end{align*}
 Again it is straightforward to check that $h_1,h_2$ are embeddings and $h_1f_1=h_2f_2$.
 This concludes the proof that $K_0^n$ has the AP for $\AA_2$ as in~\eqref{eq:ab2} and of the lemma.
\end{proof}

 We are now ready to prove the main result of this section, which yields Theorem~\ref{thm:DAG} for $\AA$ abelian.

\begin{thm} \label{thm:abag}
 Let $i\in [2]$, and let $\DD$ be the Fra{\"i}ss{\'e} limit of the class $K$ of finite direct powers of $\AA_i$
 as in~\eqref{eq:ab1}, \eqref{eq:ab2}, respectively, over a finite or countably infinite field $\FF$.
 Then $\Aut\DD$ has ample generics.
\end{thm}

\begin{proof}
Since $K_0^n$ is cofinal in $K_p^n$ for every $n\in\N$ by Lemma~\ref{lem:cofinal} and the former has the JEP and AP
by Lemma~\ref{lem:K0n}, $K_p^n$ has the JEP and CAP, consequently the WAP.
The result follows from Theorem~\ref{thm:WAP}.
\end{proof}

\begin{proof}[Proof of Theorem \ref{thm:DAG} in the abelian case]
 Recall that  every finite simple abelian Mal'cev algebra $\AA$ is term equivalent to an algebra isomorphic to one of
 $\AA_1$ or $\AA_2$ over a finite field $\FF$ by~\cite[Corollary 2.12]{Sz:CUA}.
 Hence the Fra\"iss\'e limit $\DD$ of the class of finite direct powers of $\AA$ has ample generic automorphisms
 by Theorem~\ref{thm:abag}.
 Moreover, a standard back-and-forth argument using Lemma~\ref{lem:Hrushovski} yields that $\DD$ is isomorphic
 to any countably, but not finitely generated subalgebra of the full direct power $\AA^X$ for any set $X$.
 In particular, $\DD$ is isomorphic to any filtered Boolean power $\ABxe$ and the result is proved.
\end{proof}

\begin{proof}[Proof of Corollary \ref{cor:GL}]
 Let $\FF := (F,+,\cdot)$ be a finite or countably infinite field. The countable-dimensional $\FF$-vector space
 $\V$ is the Fra{\"i}ss{\'e} limit of the finite-dimensional $\FF$-vector spaces, i.e., the finite powers of the 
 $1$-dimensional $\FF$-vector space $\AA := (F,+,(ax)_{a\in F})$, which is term equivalent to $\AA_1$ as
 in~\eqref{eq:ab1} with $m=1$ and $r=0$. Hence $\GL\V$ has ample generics by Theorem~\ref{thm:abag}. 

 Similarly, the idempotent reduct $(F,x-y+z, (ax+(1-a)y)_{a\in F})$ of $\AA$ is term equivalent to $\AA_1$
 as in~\eqref{eq:ab1} for $m=r=1$. The Fra{\"i}ss{\'e} limit of the finite direct powers is the idempotent reduct
 of $\V$, whose automorphism group $\mathrm{AGL}\, \V$ also has ample generics  by Theorem~\ref{thm:abag}.

 Finally $\PGL\V$, the quotient of $\GL\V$ by its center $Z$ of non-zero scalar transformations of $\V$,
 acts on the projective space $P(\V)$. The open sets under the topology of pointwise convergence of $\PGL\V$
 are of the form $UZ/Z$ for open sets $U$ of $\GL\V$. Further the conjugacy classes of $\PGL\V$ are of the form
 $CZ/Z$ for conjugacy classes $C$ of $\GL\V$. Hence any generic $n$-tuple $(g_1,\dots,g_n)$ in $\GL\V$ yields
 a generic tuple  $(g_1Z,\dots,g_nZ)$ in $\PGL\V$. Thus $\PGL\V$ has ample generics.
\end{proof}

\section{Projective Fra{\"\i}ss\'e theory}
\label{sec:PFT}

We recall basic notions and results on the projective Fra{\"\i}ss\'e theory developed by Irwin and
Solecki~\cite{IS:PFL} as needed for our proof of Theorem~\ref{thm:H}.
Let $\sigma$ be a signature (language) consisting of relation symbols $R_i$ with arity $m_i\in\N$ for $i\in I$ and function symbols
$f_j$ with arity $n_j \geq 0$ for $j \in J$.
A \emph{topological $\sigma$-structure} $\A$ is a zero-dimensional, compact, second countable space $A$ together with closed sets
$R_i^\A \subseteq A^{m_i}$ for all $i \in I$ and continuous functions $f_j^\A\colon A^{n_j}\to A$ for all $j \in J$.
For a function symbol $f_j$ of arity $0$, the corresponding function $f_j^\A$ is simply interpreted as a constant
in $A$. Throughout, we will use blackboard bold letters to denote topological $\sigma$-structures
and matching standard letters for the underlying topological spaces.

 For topological $\sigma$-structures $\A,\B$, an \emph{epimorphism} $\phi\colon\A\to\B$ is 
 a surjective continuous map $A\to B$
 such that for every function symbol $f$ with arity $n\geq 0$ in $\sigma$ and
 for all 
 $\bx\in A^n$
 we have
\[ \phi(f^\A(\bx)) = f^\B(\phi(\bx)) \]
 and for every  relation symbol $R$ in $\sigma$ we have 
 \[ \phi(R^\A) = R^\B, \]
 where $\phi$ is applied componentwise on the tuples in $R^\A$.

A family $\F$ of topological $\sigma$-structures is a \emph{projective Fra{\"\i}ss\'e family} if it satisfies
the following:
\begin{enumerate}[label={(JPP)}]
\item \label{JPP}
  (Joint Projection Property) For any $\A_1,\A_2\in\F$ there exist ${\B\in\F}$ and epimorphisms $\psi_1,\psi_2$
  from $\B$ onto $\A_1$ and $\A_2$, respectively.

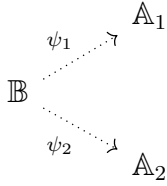
\begin{center}
  \begin{tikzcd}[row sep=small]
    & \A_1 \\
\B \arrow[dotted]{ru}{\psi_1} \arrow[dotted]{rd}{}[swap]{\psi_2} & \\
& \A_2
\end{tikzcd}
\end{center}
\end{enumerate}

\begin{enumerate}[label={(PAP)}]
\item\label{AP}
(Projective Amalgamation Property) For $\A_1,\A_2,\B\in\F$ and epimorphisms $\phi_1\colon\A_1\to\B$ and $\phi_2\colon\A_2\to\B$,
there exist $\C\in\F$ and epimorphisms $\psi_1\colon \C\to\A_1$ and $\psi_2\colon \C\to\A_2$ such that
$\phi_1\psi_1=\phi_2\psi_2$.

\begin{center}
  \begin{tikzcd}[row sep=small]
    & \A_1 \arrow[rd]{}{\phi_1} & \\
\C \arrow[dotted]{ru}{\psi_1} \arrow[dotted]{rd}{}[swap]{\psi_2} & & \B \\
& \A_2 \arrow[ru]{}[swap]{\phi_2} &
\end{tikzcd}
\end{center}
\end{enumerate}

A topological $\sigma$-structure $\bL$ is a \emph{projective Fra{\"\i}ss\'e limit} of a \emph{projective Fra{\"\i}ss\'e family} $\F$ if it satisfies the following:
\begin{enumerate}[label={(L\arabic*)}]
\item\label{it:L1}
(Projective Universality) For any $\A\in\F$ there exists an epimorphism $\bL\rightarrow \A$.
\item\label{it:L2}
(Factoring Property)
For any finite discrete topological space $A$ and any continuous function $f\colon \L\rightarrow A$
there exist $\B\in\F$, an epimorphism $\psi\colon \bL\rightarrow \B$ and a function $g\colon B\rightarrow A$ such that $f=g\psi$.

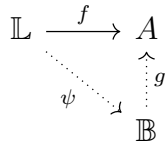
\begin{center}
\begin{tikzcd}
\bL  \arrow{r}{}{f} \arrow[dotted]{rd}[swap]{\psi}& A \\
& \B \arrow[dotted]{u}[swap]{g} 
\end{tikzcd}
\end{center}
\item\label{it:L3}
(Projective Homogeneity)
For any $\A\in\F$ and any epimorphisms $\phi_1,\phi_2\colon\bL\rightarrow \A$ there exists an automorphism $\alpha\colon\bL\rightarrow\bL$ such that $\phi_1=\phi_2\alpha$.

\begin{center}
\begin{tikzcd}
\bL \arrow[dotted]{d}{}[swap]{\alpha} \arrow{r}{}{\phi_1} & \A \\
\bL \arrow{ru}[swap]{\phi_2} &
\end{tikzcd}
\end{center}
\end{enumerate}

The next result of Irwin and Solecki is fundamental for Kwiatkowska's approach of constructing generics on $2^\omega$
as projective Fra{\"\i}ss\'e limits in~\cite{Kw:GHC}, which we modify for our purposes in Section~\ref{sec:X0}.

\begin{thm}[{\cite[Theorem 2.4]{IS:PFL}}]
\label{thm:IrwSol}
Every countable projective Fra{\"\i}ss\'e family of finite topological $\sigma$-structures has a projective Fra{\"\i}ss\'e limit, and any two such limits are isomorphic.
\end{thm}

 In the above, a family is understood to be countable if it contains only countably many structures up to isomorphism.
 This clearly holds when $\sigma$ is finite.

 There is a further property that every projective Fra{\"\i}ss\'e limit possesses, which can be used as an
 alternative for \ref{it:L3}:

\begin{enumerate}[label={(L\arabic*)}]
\setcounter{enumi}{3}
\item
\label{it:L4}
(Extension Property)
For any $\A,\B\in\F$, and any epimorphisms $\phi_1\colon \bL\rightarrow \A$ and $\phi_2\colon \B\rightarrow \A$,
there exists an epimorphism $\psi\colon \bL\rightarrow \B$ such that $\phi_1=\phi_2\psi$.

\begin{center}
\begin{tikzcd}
\bL \arrow[dotted]{d}{}[swap]{\psi} \arrow{r}{}{\phi_1} & \A \\
\B \arrow{ru}[swap]{\phi_2} &
\end{tikzcd}
\end{center}
\end{enumerate}

\begin{prop}[{\cite[Proposition 2.2.]{Kw:GHC}}]
  \label{pr:EP}
 Let $\F$ be a countable projective Fra{\"\i}ss\'e family. 
 The projective Fra{\"\i}ss\'e limit of $\F$ satisfies \ref{it:L4}. Furthermore, if a structure satisfies
 \ref{it:L1}, \ref{it:L2} and \ref{it:L4}, then it also satisfies \ref{it:L3} and is therefore isomorphic
 to the projective Fra{\"\i}ss\'e limit of $\F$.
\end{prop}

\section{Pointwise stabilisers in $\Homeo 2^\omega$}
\label{sec:X0}

 This section consists of the proof of Theorem~\ref{thm:H} by obtaining generic $n$-tuples as projective
 Fra{\"\i}ss\'e limits of certain families as in~\cite{Kw:GHC}.
 For the formulation of our characterization of generics in Theorem~\ref{thm:H2} we need to
 introduce some notation. 

Let $\sigma$ be the signature $\{ s_1,\dots, s_m\}$ consisting of $m$ binary relation symbols.
Following~\cite{Kw:GHC} we say that a binary relation $s$ on a set $A$ is \emph{surjective}
if for every $a\in A$ there exist $b,c\in A$ such that ${(a,b),(c,a)\in s}$.  
An element $a\in A$ is \emph{outgoing} for $s$ if
\[  |\{ b\in A \st (b,a)\in s\}| = 1 \quad \text{and} \quad |\{b\in A \st (a,b)\in s \}| \geq 2. \]
   
We use the following two families of $\sigma$-structures from \cite{Kw:GHC}:
\begin{itemize}
\item
 $\F_0$ consists of all \emph{surjective}  finite topological $\sigma$-structures, i.e.~structures $\A=(A;s_1^\A,\dots,s_m^\A)$
 where all $s_i^\A$ are surjective.
\item
$\F$ consists of all structures in $\F_0$ which additionally satisfy the following two properties:
\begin{enumerate}
\item
every point in $A$ is outgoing for precisely one of the relations $s_1^\A, (s_1^\A)^{-1},\dots, s_m^\A,(s_m^\A)^{-1}$;
\item
if $s$ is any of $s_1,\dots,s_m$ and $(a,b)\in s^\A$, then $a$ is $s^\A$-outgoing or $b$ is $(s^\A)^{-1}$-outgoing.
\end{enumerate}
\end{itemize}

To a $\sigma$-structure $\A$ we can associate a simple graph with vertices $A$ and edges $\{a,b\}$ for any $a,b\in A$ such that
$a\neq b$ and $(a,b)\in s_i^\A$ for some $i\in [m]$.
The connected components of this graph will be referred to as the \emph{connected components} of the structure $\A$.
Note that $\A\in \F$ (resp.\ $\A\in\F_0$) if and only if for every connected component $C$ of $\A$
the induced substructure $\C$ is in $\F$ (resp. in $\F_0$).

We now list the key properties of the families $\F$ and $\F_0$ that we need: 

\begin{prop} \label{pr:Kw} 
The following hold for families $\F,\F_0$ of $\sigma$-structures defined above:
\begin{enumerate}
\item \cite[Theorem 4.6]{Kw:GHC}
\label{it:Kw1}
$\F$ is coinitial in $\F_0$, i.e., for every $\A\in\F_0$ there exist $\B\in\F$ and an epimorphism $\B\to\A$.
\item \cite[cf. Theorem 4.1]{Kw:GHC}
\label{it:Kw2}
$\F$ is a countable projective Fra\"\i ss\'e family.
\item \cite[Theorem 4.7]{Kw:GHC}
\label{it:Kw3}
If $\bL=(\L; s_1^\bL,\dots ,s_m^\bL) $ is the projective Fra{\"\i}ss\'e limit of $\F$, then $\L$ is homeomorphic to the Cantor space, the relations $s_1^\bL,\dots,$ $s_m^\bL$ are graphs of homeomorphisms of $\L$,
and $(s_1^\bL,\dots, s_m^\bL)$ is a generic $m$-tuple for $\Homeo \L$.
\end{enumerate}
\end{prop}

Let $\sigma^{(n)}$ denote the expansion of the signature $\sigma$ by $n$ new constant symbols $p_1,\dots,p_n$.
Let $\F_0^{(n)} $ be the class of all $\sigma^{(n)}$-structures $\A=(A; s_1^\A,\dots,s_m^\A,p_1^\A,\dots,p_n^\A)$
such that $(A; s_1^\A\dots,s_m^\A)\in \F_0$ and ${(p_j^\A,p_j^\A)\in s_i^\A}$ for all $i\in [m]$, $j\in [n]$.

For $\A=(A;  s_1^\A,\dots,s_m^\A)\in\F_0$ define
\[ \A^{(n)}:=(A^{(n)}; s_1^{\A^{(n)}},\dots,s_m^{\A^{(n)}},p_1^{\A^{(n)}},\dots,p_n^{\A^{(n)}}) \]
where:
\begin{alignat*}{3}
A^{(n)} & := A\dotcup [n], &\makebox[5mm]{}&\\
s_i^{\A^{(n)}}&:= s_i^\A\cup \Delta_{[n]} && \text{for } i\in [m],\\
p_j^{\A^{(n)}} &:= j && \text{for } j\in [n].
\end{alignat*}
Then
\[ \F^{(n)}:=\{\A^{(n)}\colon \A\in \F\} \subseteq \F_0^{(n)}. \]

We are going to prove the following generalizations of the facts in Proposition~\ref{pr:Kw}:
\begin{itemize}
\item
$\F^{(n)}$ is coinitial in $\F_0^{(n)}$ (Lemma \ref{la:Focoinit}).
\item
$\F^{(n)}$ is a projective Fra{\"\i}ss\'e class, so that it has a projective Fra{\"\i}ss\'e limit
$\bL=(\L;s_1^{\bL},\dots,s_m^\bL,p_1^\bL,\dots,p_n^\bL)$ (Lemma \ref{la:FoPFF}).
\item
$\L$ is homeomorphic to the Cantor space (Lemma \ref{la:LCantor}).
\item
Each $s_i^\bL$ for $i\in[m]$ is the graph of a homeomorphism of $\L$ which fixes all $p_j^\bL$ for $j\in [n]$ (Lemma \ref{la:graphhomeo}).
\item
The orbit of $(s_1^\bL,\dots,s_m^\bL)$ under the diagonal conjugation action  
of $H:=(\Homeo \L)_{(p_1^\bL,\dots,p_n^\bL)}$ is comeager in $H^m$ (Lemmas~\ref{la:dense}, \ref{la:Gdelta}).
\end{itemize}

This immediately yields Theorem~\ref{thm:H}. Additionally we will obtain the following
characterization of generics in $H^m$ as the main result of this section.

\begin{thm} \label{thm:H2}
 Let $m\in\N, n\in\N\cup\{0\}$, let $x_1,\dots,x_n\in 2^\omega$ be distinct,
 and let $H:=\bigl(\Homeo 2^\omega\bigr)_{(x_1,\dots,x_n)}$. 
 For $h_1,\dots,h_m\in H$ the following are equivalent:
\begin{enumerate}
\item \label{it:H1}
 The $m$-tuple $(h_1,\dots,h_m)$ is generic for $H$;
\item \label{it:H2}
  $(2^\omega;h_1,\dots,h_m,x_1,\dots,x_n)$ is isomorphic to the projective Fra{\"\i}ss\'e limit $\bL$ of $\F^{(n)}$
  over the signature $\{s_1,\dots,s_m,p_1,\dots,p_n\}$.
\end{enumerate}
\end{thm}

 We start with establishing the necessary properties of $\F^{(n)}$.

\begin{lemma}
\label{la:Focoinit}
$\F^{(n)}$ is coinitial in $\F_0^{(n)}$.
\end{lemma}

\begin{proof}
 For $\A\in\F_0^{(n)}$, let $\A' := (A;s_1^\A,\dots,s_m^\A)$ be the reduct of $\A$ obtained by removing the constant symbols $p_j$ for $j\in [n]$
 from the signature $\sigma^{(n)}$. Note that $\A'\in \F_0$.
By  Proposition~\ref{pr:Kw}\eqref{it:Kw1} there exist $\B\in\F$ and an epimorphism $\phi\colon \B\rightarrow \A'$.
The extension of $\phi$ to $\B^{(n)}\in\F^{(n)}$ setting $ k\mapsto p_k^\A$ for $k\in [n]$ is an epimorphism ${\B^{(n)}\rightarrow \A}$, as required.
\end{proof}

\begin{lemma}
\label{la:FoPFF}
$\F^{(n)}$ is a countable projective Fra{\"\i}ss\'e family.
\end{lemma}

\begin{proof}
That $\F^{(n)}$ is countable follows from finiteness of signature $\sigma^{(n)}$.
We now check that $\F^{(n)}$ satisfies \ref{JPP} and \ref{AP}.

For \ref{JPP}, letting $\A_1,\A_2\in\F$, use Proposition \ref{pr:Kw}\ref{it:Kw2}
to derive
  $\B\in\F$ and epimorphisms $\B\rightarrow\A_1$ and $\B\rightarrow \A_2$.
  Then extend these epimorphisms to $\B^{(n)}\rightarrow \A_1^{(n)}$ and $\B^{(n)}\rightarrow \A_2^{(n)}$
  by setting $k\mapsto k$ for $k\in [n]$. 

For \ref{AP}, let $\A,\B_1,\B_2\in\F$ and $\phi_i \colon \B_i^{(n)}\rightarrow \A^{(n)}$ be epimorphisms for $i\in [2]$.
We need to show that there exist $\C\in\F$ and epimorphisms $\psi_i\colon \C^{(n)}\rightarrow \B_i^{(n)}$, $i\in [2]$, such that
$\phi_1\psi_1=\phi_2\psi_2$.

Let $i\in [2]$.
Since each $\{p_j^{\A^{(n)}}\}$ is a connected component of $\A^{(n)}$, it follows that $\phi_i^{-1} (p_j^{\A^{(n)}})$
is a union of connected components of $\B_i^{(n)}$.
Hence $B_i':= B_i\setminus \phi_i^{-1}(\{ p_1^{\A^{(n)}},\dots, p_n^{\A^{(n)}}\})$ is a non-empty union of
connected components of $\B_i$. It follows that there exists a substructure $\B_i'$ of $\B_i$
with universe $B_i'$, and that $\B_i'\in \F$.
The restriction $\phi_i'$ of $\phi_i$ to $B_i$ is an epimorphism $\B_i'\rightarrow \A$.

By Proposition \ref{pr:Kw}\ref{it:Kw2}, there exist $\C\in\F$ and epimorphisms $\psi_i'\colon \C\rightarrow \B_i'$ for
$i\in [2]$ such that $\phi_1'\psi_1'=\phi_2'\psi_2'$.
But then $\psi_i:=\psi_i'\cup \Delta_{[n]}$ is an epimorphism $\C^{(n)}\rightarrow \B_i^{(n)}$ and
$\phi_1\psi_1=\phi_2\psi_2$, as required. Hence $\F^{(n)}$ satisfies \ref{AP}, completing the proof of the lemma.
\end{proof}

For the remainder of the paper we denote the projective Fra{\"\i}ss\'e limit of $\F^{(n)}$ by
\[
\bL := (\L; s_1^\bL,\dots, s_m^\bL,p_1^\bL,\dots,p_n^\bL).
\]

\begin{lemma}
\label{la:LCantor}
$\L$ is homeomorphic to the Cantor space.
\end{lemma}

\begin{proof}
This closely follows the argument in \cite[Proposition 3.5]{Kw:GHC}.
From the general projective Fra{\"\i}ss\'e  theory we know that $\bL$ is a topological structure and, in particular, that $\L$ is compact, zero-dimensional and second-countable.
It remains to show that $\L$ has no isolated points.

The reader may find it helpful to refer from time to time to Figure~\ref{fig:LCantor},
which shows the structures and mappings featuring in the following argument.
Aiming for a contradiction, suppose that $p$ is an isolated point in $\L$.
This means that $\{p\}$ is open, and, by zero-dimensionality, so is $\L\setminus\{p\}$.
Then the projection $f\colon \L\rightarrow \bigl\{\{p\},\L\setminus\{p\}\bigr\}$ is continuous.
By the Factoring Property \ref{it:L2}, we have $\A^{(n)}\in\F^{(n)}$,
an epimorphism $\phi_1\colon \bL\rightarrow \A^{(n)}$, and a mapping 
$g\colon A\cup [n]\rightarrow \bigl\{\{p\},\L\setminus\{p\}\bigr\}$, such that $g\phi_1=f$.
Let $a_0:=\phi_1(p)\in A\cup [n]$.

Let $\B$ be the disjoint union of two copies of $\A$, say $\B=\A\cup \A'$, with $A'=\{ a'\colon a\in A\}$.
Since $\A\in\F$ we also have $\B\in \F$.
Define a homomorphism
\[ \phi_2\colon \B^{(n)}\rightarrow \A^{(n)},\ x\mapsto \begin{cases} x & \text{if } x\in A\cup [n], \\
   a_0 &\text{if } x\in A'\text{ and } a_0\in [n], \\                                                         a &\text{if } x=a'\in A' \text{ and } a_0\in A.\end{cases}
\]

Regardless of whether $a_0$ is in $[n]$ or in $A$,
the element $a_0$ has at least two preimages under $\phi_2$.
Hence $p$ has at least two preimages under $g\phi_2$.
By the Extension Property \ref{it:L4}, there is an epimorphism $\psi\colon \bL\to \B^{(n)}$ such that
$\phi_1 = \phi_2\psi$.
It now follows that $p$ has at least two preimages under $g\phi_2\psi=f$.
This is a contradiction, as $f^{-1}(p)=\{p\}$ by assumption, proving  that $\L$ has no
isolated points, as required.
\begin{figure} 
\begin{center}
\begin{tikzcd}
\bL  \arrow{r}{}{f} \arrow{rd}{}[swap]{\phi_1} \arrow{d}{}[swap]{\psi} & \bigl\{ \{p\}, L\setminus\{L\} \bigr\} \\
\B^{(n)} \arrow{r}{}[swap]{\phi_2}& \A^{(n)} \arrow{u}[swap]{g} 
\end{tikzcd}
\end{center}
\caption{Structures and mappings in the proof of Lemma \ref{la:LCantor}.}
\label{fig:LCantor}
\end{figure}
\end{proof}

\begin{lemma}
  \label{la:clcoco}
 Let $i\in [n]$ and let $b$ be a clopen of $\L$ such that $p_i^\bL\in b$. Then
 there exists a clopen $c$ such that $p_i^\bL\in c\subseteq b$ and $c$ is a union of connected components of $\bL$.
\end{lemma}

\begin{proof}
 Figure \ref{fig:clcoco} shows the structures and mappings featuring in the argument below.
 Since both $b$ and $\L\setminus b$ are open, the projection $f\colon \L\rightarrow \{ b,\L\setminus b\}$
 is continuous.
 By the Factoring Property \ref{it:L2}, there exist $\A^{(n)}\in \F^{(n)}$, an epimorphism $\phi_1\colon\bL\rightarrow \A^{(n)}$,
 and a (necessarily surjective) mapping $g\colon A\cup [n]\rightarrow  \{ b,\L\setminus b\}$ such that
 $f = g\phi_1$.
 Note that the partition $\{ \phi_1^{-1}(a)\colon a\in A\cup [n]\}$ is a refinement of the partition
 $\{ b,\L\setminus b\}$.
 Now define a new structure $\B\in\F_0^{(n)}$ to be an extension of $\A^{(n)}$ by a new element $0$ with the constant
 $p_i^\B$ redefined as $0$, that is,
\begin{align*}
B &:= A\dotcup\{0,\dots,n\}, \\
s_j^\B &:= s_j^{\A^{(n)}}\cup \{(0,0)\} \text{ for } j\in [m],\\
  p_j^\B &:= \begin{cases} 0 & \text{if } j=i, \\
               j &\text{if } j\in [n]\setminus\{i\}. \end{cases}
\end{align*}
 Note that
\[
\psi_1\colon \B\rightarrow \A^{(n)},\ x\mapsto \begin{cases} i & \text{if } x=0, \\ x & \text{else}, \end{cases}    
\]
 is an epimorphism.
 By coinitiality (Lemma~\ref{la:Focoinit})
 there exist $\C^{(n)}\in\F^{(n)}$ and an epimorphism $\psi_2\colon \C^{(n)}\rightarrow \B$.
 The Extension Property \ref{it:L4} for $\A^{(n)},\C^{(n)}\in\F^{(n)}$, $\phi_1\colon \bL\rightarrow \A^{(n)}$ and
 $\psi_1\psi_2\colon \C^{(n)}\rightarrow \A^{(n)}$ yields a homomorphism $\psi_3\colon \bL\rightarrow \C^{(n)}$
 such that $\psi_1\psi_2\psi_3=\phi_1$.

We claim that $c:=\psi_3^{-1}\psi_2^{-1}(0)$ has the required properties.
Indeed:
\begin{itemize}
\item
$c$ is clopen because $\psi_2\psi_3$ is continuous;
\item
$c$ is a union of connected components of $\bL$ because $\{0\}$ is a connected component in $\B$ and $\psi_2\psi_3$  is a homomorphism;
\item
 $c\subseteq b$ because the preimages under $\psi_3$ are a refinement of the partition $\{b,L\setminus b\}$,
 and $f(p_i^\bL)= b$.\qedhere
\end{itemize}
\begin{figure} 
\begin{center}
\begin{tikzcd}
\bL  \arrow{rr}{}{f} \arrow{rrd}{\phi_1} \arrow{d}{}[swap]{\psi_3} && \bigl\{ b, L\setminus b \bigr\} \\
\C^{(n)} \arrow{r}{}[swap]{\psi_2} & \B \arrow{r}{}[swap]{\psi_1}& \A^{(n)} \arrow{u}[swap]{g} 
\end{tikzcd}
\caption{Structures and mappings in the proof of Lemma \ref{la:clcoco}.}
\label{fig:clcoco}
\end{center}
\end{figure}
\end{proof}

\begin{lemma}
\label{la:MLk}
Let $\M$ be a non-empty clopen union of connected components of $\bL$
such that $\M\cap \{p_1^\bL,\dots,p_n^\bL\} = \{p_{j_1}^\bL,\dots,p_{j_k}^\bL\}$ 
with $0\leq k\leq n$ and $j_1,\dots, j_k$ distinct.
 Then the structure 
 \[
\bM:=(\M; s_1^\bL\restr_\M,\dots, s_m^\bL\restr_\M, p_{j_1}^\bL,\dots,p_{j_k}^\bL)
 \]
 is isomorphic to 
the projective Fra{\"\i}ss\'e limit of $\F^{(k)}$.
\end{lemma}

\begin{proof}
Without loss of generality we can assume that $j_t=t$ for all $t\in [k]$.
We will show that the structure $\bM$ satisfies 
Projective Universality \ref{it:L1}, the Factoring Property \ref{it:L2} and the Extension Property \ref{it:L4}
with respect to the family $\F^{(k)}$, and the lemma will follow by an application of Proposition~\ref{pr:EP}.

\noindent
\textbf{Projective Universality \ref{it:L1}.}
The structures and mappings featuring in this part of the proof are shown in Figure \ref{fig:PUMLk}.
\begin{figure} 
\begin{center}
\begin{tikzcd}
& \bL  \arrow{r}{}{\phi_1} \arrow{rdd}{\psi_5} \arrow{d}{}[swap]{\psi_1} \arrow{ld}{}[swap]{\phi_2}& \A^{(n)} \\
\P & \B^{(n)} \arrow{d}{}[swap]{\phi_3} \arrow{l}{}[swap]{\psi_2} & \\
 & \C^{(n)} \arrow{lu}{\phi_4} & \D^{(n)} \arrow{l}{\psi_4} \arrow{uu}{}[swap]{\psi_3}
\end{tikzcd}
\end{center}
\caption{Structures and mappings in the proof of Projective Universality in Lemma \ref{la:MLk}.}
\label{fig:PUMLk}
\end{figure}
Let $\A\in\F(=\F^{(0)})$. We will show that there exists an epimorphism $\bM\to\A^{(k)}$.
For this note that $\A^{(n)}\in\F^{(n)}$ contains $\A^{(k)}$ as a substructure of a reduct.
By the Projective Universality of $\bL$ there exists an epimorphism $\phi_1\colon \bL\rightarrow \A^{(n)}$.

Let 
\[
\P:= \bigl(\{0,\dots,n\};\Delta_P,\dots,\Delta_P,1,\dots,n\bigr)\in\F_0^{(n)}.
\]
We claim that there exists an epimorphism $\phi_2\colon \bL\rightarrow \P$ such that $\phi_2^{-1}(\{0,\dots,k\})=\M$.
To see this, let $\{ b_i\colon i\in\{0,\dots,n\}\}$ be a partition of $\L$ into non-empty clopens such that
$p_i^\bL\in b_i$ for all $i\in [n]$ and $\bigcup_{i=0}^k b_i=\M$.
By Lemma~\ref{la:clcoco}, for each $i\in [n]$ there exists a clopen $c_i$ such that $p_i^\bL\in c_i\subseteq b_i$
and $c_i$ is a union of connected components of $\bL$.
Without loss of generality we may assume that $\bigcup_{i=k+1}^n c_i = \L\setminus\M$; if not, we just replace $c_n$
by $\L\setminus \bigl(\M\cup\bigcup_{i=k+1}^{n-1} c_{i}\bigr)$, which is still a union of connected components since
$\M$ is. Also, define $c_0:= \M\setminus \bigcup_{i=1}^k c_i$,
and note that $\emptyset\neq b_0\subseteq c_0$ and $\bigcup_{i=0}^k c_i=\M$.
Since all $c_i$ are unions of connected components, the mapping
\[
\phi_2\colon \L\rightarrow \{0,\dots,n\},\ x\mapsto i \text{ for } x\in c_i,
\]
is the claimed epimorphism $\bL\rightarrow \P$.  

By the Factoring Property \ref{it:L2} for $\bL$ there exist $\B^{(n)}\in\F^{(n)} $,
an epimorphism $\psi_1\colon \bL\rightarrow \B^{(n)}$, and a function $\psi_2\colon B^{(n)}\rightarrow P$ such that
$\phi_2=\psi_2\psi_1$.
In fact, it is easy to see that $\psi_2$ is also an epimorphism, because $\phi_2$ and $\psi_1$ are.

For every $i\in [n]$ we have $\psi_2^{-1}(i)=\{ i\}\dotcup B_i$ with $B_i$ either empty or else the universe of a
structure $\B_i\in \F$. Likewise, $C:=\psi_2^{-1}(0)$, which contains $0$, is the universe of a structure $\C\in\F$.

It is routine that
\begin{alignat*}{2}
\phi_3&\colon \B^{(n)}\rightarrow \C^{(n)},\ & x&\mapsto\begin{cases} x&\text{if } x\in C, \\ i&\text{if } x\in \{i\}\cup B_i,\end{cases}
\\
\phi_4&\colon \C^{(n)}\rightarrow \P,& x&\mapsto \begin{cases} 0 & \text{if } x\in C, \\ i & \text{if } x=i,   \end{cases}
\end{alignat*}
are epimorphisms and that $\psi_2 = \phi_4\phi_3$.

Now, the structures $\A$ and $\C$ are both in $\F$, and we can use \ref{JPP} to obtain $\D\in\F$ and epimorphisms
$\D\rightarrow \A$ and $\D\rightarrow \C$. These epimorphisms can be extended with the identity on $[n]$
to epimorphisms $\psi_3\colon \D^{(n)}\rightarrow \A^{(n)}$ and $\psi_4\colon\D^{(n)}\rightarrow \C^{(n)}$.
By Projective Universality \ref{it:L1}, there exists an epimorphism $\psi_5\colon\bL\rightarrow \D^{(n)}$.

 Now $\phi_3\psi_1$ and $\psi_4\psi_5$ both are epimorphisms $\bL\rightarrow \C^{(n)}$.
 By Projective Homogeneity \ref{it:L3} for $\bL$, there exists $\alpha\in \Aut\bL$ such that
 $\phi_3\psi_1=\psi_4\psi_5\alpha$.
 Since $\psi_3^{-1}(i) = \{ i \} = \psi_4^{-1}(i)$ for all $i\in [n]$,
 we see that 
\begin{align*}
\alpha^{-1}\psi_5^{-1}\psi_3^{-1}([n]\setminus [k]) 
 &= \alpha^{-1}\psi_5^{-1}([n]\setminus [k])=\alpha^{-1}\psi_5^{-1}\psi_4^{-1}([n]\setminus [k])\\
 &= \psi_1^{-1}\phi_3^{-1}([n]\setminus [k])=\L\setminus \M.
\end{align*}
 Therefore
 $(\psi_3\psi_5\alpha)^{-1}(A\cup[k])=\M$, and so $\psi_3\psi_5\alpha$ restricts to an epimorphism
 $\bM\rightarrow \A^{(k)}$.

\noindent
\textbf{Factoring property \ref{it:L2}.}
The structures and mappings featuring in this part of the proof are shown in Figure \ref{fig:FPMLk}.
\begin{figure} 
\begin{center}
\begin{tikzcd}
 \bM \arrow[d, phantom, sloped, "\subseteq"]  \arrow{r}{}{f_0} & A  \arrow[d, phantom, sloped, "\subseteq"]\\
\bL  \arrow{r}{}{f} \arrow{rd}{}[swap]{\psi} \arrow{d}{}[swap]{\psi_1} & A\dotcup\{k+1,\dots,n\} \\
 (\B\dotcup\C)^{(n)} \arrow{r}{}[swap]{\phi} & \B^{(n)} \arrow{u}{}{g} \\
\end{tikzcd}
\end{center}
\caption{Structures and mappings in the proof of the Factoring Property in Lemma \ref{la:MLk}.}
\label{fig:FPMLk}
\end{figure}
Let $A$ be a finite discrete space and $f_0\colon \M\rightarrow A$ be a continuous map.
Let $f\colon \L\rightarrow A\dotcup\{k+1,\dots,n\}$
be any continuous extension of $f_0$
satisfying
\[
f(\L\setminus \M)=\{k+1,\dots,n\} \quad\text{and}\quad
f(p_i^\bL)=i\text{ for } i\in \{k+1,\dots,n\}.
\]
Such an $f$ exists by using Lemma \ref{la:clcoco} to partition $\L\setminus \M$ into clopens
$b_{k+1},\dots,b_{n}$ such that $b_i$ is a union of connected components and $p_i^\bL\in b_i$ for all
$i\in\{k+1,\dots,n\}$.
By \ref{it:L2} for $\bL$, there exist $\B^{(n)}\in\F^{(n)}$,
an epimorphism $\psi\colon \bL\rightarrow \B^{(n)}$, and a mapping
$g\colon B\cup [n]\rightarrow A\cup\{k+1$, $\dots,n\}$ such that
$g\psi=f$.
It immediately follows that $\psi(\M)$ is a union of connected components of $\B^{(n)}$
and that $\psi(\M) \cap [n] = [k]$.

If $\psi(\M)\cap B\neq \emptyset$, then $\psi(\bM)\in\F^{(k)}$ and 
\[
(g\restr_{\psi(\M)})(\psi\restr_\M)=f\restr_\M=f_0
\]
 gives the required property for \ref{it:L2}.

Consider now the case where $\psi(\M)\cap B=\emptyset$. Note that this can happen only when $k\neq 0$.
In that case, let $\C\in\F$ be arbitrary. The structure $(\B\dotcup\C)^{(n)}$ maps homomorphically onto
$\B^{(n)}$ via 
\[
\phi\colon x\mapsto\begin{cases} x & \text{if } x\in B^{(n)},\\ 1&\text{if } x\in C. \end{cases}
\]
By the Extension Property for $\bL$, there exists a homomorphism
$\psi_1\colon \bL\rightarrow (\B\dotcup\C)^{(n)}$ such that $\phi\psi_1=\psi$.
Now $C\subseteq \psi_1(\M)$ and the previous case can be applied with $\psi_1$ and $(\B\dotcup\C)^{(n)}$
instead of $\psi$ and $\B^{(n)}$.

\noindent
\textbf{Extension Property \ref{it:L4}.}
 The structures and mappings featuring in this part of the proof are shown in Figure \ref{fig:EPMLk}.
\begin{figure} 
\begin{center}
\begin{tikzcd}
  \bM \arrow{rr}{}{\phi_1} \arrow{rdd}{}{} &&  \A^{(k)} \\
  \bL \arrow[u, phantom, sloped, "\supseteq"] \arrow{rr}{}{\phi'_1} \arrow{rdd}{}[swap]{\psi} &&  \A^{(n)} \arrow[u, phantom, sloped, "\supseteq"] \\
 & \B^{(k)} \arrow{ruu}{}{\phi_2} & \\
 & \B^{(n)} \arrow[u, phantom, sloped, "\supseteq"] \arrow{ruu}{}[swap]{\phi'_2} & \\
\end{tikzcd}
\end{center}
\caption{Structures and mappings in the proof of the Extension Property in Lemma \ref{la:MLk}.}
\label{fig:EPMLk}
\end{figure}
Let $\A^{(k)},\B^{(k)}\in\F^{(k)}$ and let $\phi_1\colon \bM\rightarrow \A^{(k)}$,
 $\phi_2\colon \B^{(k)}\rightarrow \A^{(k)}$,
 be epimorphisms.
 Using the same idea based on Lemma~\ref{la:clcoco} as in the previous part, we extend $\phi_1$ to an
 epimorphism $\phi_1'\colon \bL\rightarrow \A^{(n)}$ such that
\[ \phi_1'(\L\setminus \M)=\{k+1,\dots,n\} \text{ and } \phi_1'(p_i^\bL)=i \text{ for } i\in \{k+1,\dots,n\}. \]
We also extend $\phi_2$ to an epimorphism
\[ \phi_2'\colon \B^{(n)}\rightarrow \A^{(n)} \text{ via } x\mapsto x \text{ for } x\in\{k+1,\dots,n\}. \]
By the Extension Property \ref{it:L4} of $\bL$ there exists an epimorphism $\psi\colon \bL\rightarrow\B^{(n)}$
such that $\phi_1' = \phi_2'\psi$.
We note that $\psi(\bM) = \B^{(k)}$ since
for $x\in\bL$:
\[
\psi(x)\in [n]\setminus [k] \Leftrightarrow \phi'_1\psi(x)\in [n]\setminus [k]\Leftrightarrow \phi'_2(x)\in [n]\setminus k \Leftrightarrow x\in \L\setminus\M.
\]
Since $\M$ is a union of connected components, it follows that $\psi$ restricts to an epimorphism
$\bM\to\B^{(k)}$.
Observing that $\phi_1 = \phi_2\psi\restr_\M$ completes the proof of the Extension Property and of the lemma.
\end{proof}

\begin{lemma}
\label{la:graphhomeo}
$s_1^\bL,\dots, s_m^\bL$ are graphs of homeomorphisms of the Cantor space $\L$ fixing $p_1^\bL,\dots,p_n^\bL$.
\end{lemma}

\begin{proof}
 We claim that $\L\setminus \{ p_1^\bL,\dots, p_n^\bL\}$ can be partitioned into clopens $b_{ij}$ for
 $i\in [n]$, $j\in\N$, which are unions of connected components such that each sequence
 $(b_{ij})_{j\in\N}$ converges to $p_i^\bL$ and $\bigcup_{j\in\N} b_{ij} \cup \{p_i^\bL\}$ is a union of
 connected components for each $i\in [n]$.
 To see this, first use Lemma \ref{la:clcoco} to 
 partition $L$ into clopens $b_1,\dots,b_n$ such that each $b_i$ is a union of connected components and contains $p_i^\bL$.
Now, for $i\in [n]$, use Lemma \ref{la:clcoco} again to find a decreasing sequence of clopens
$b_1=c_{i1}\supset c_{i2}\supset c_{i3}\supset \dots$ containing $p_i^\bL$ which converges to $p_i^\bL$
 where $c_{ij}$ for $j\in\N$ are unions of connected components.
Letting $b_{ij}:=c_{ij}\setminus c_{i,j+1}$ for $i\in [n]$ and $j\in\N$ yields the desired partition.
At this point we know that each $\{p_i^\bL\}$ is a singleton connected component in $\bL$.
 
 By Lemma \ref{la:MLk}, the restriction of $\bL$ to any $b_{ij}$ is isomorphic to the projective
 Fra{\"\i}ss\'e limit of $\F=\F^{(0)}$.
By \cite[Proposition 4.9]{Kw:GHC}, each $s_k^\bL\restr_{b_{ij}}$ is a homeomorphism of $b_{ij}$.
 Since $b_{ij}$ are unions of connected components, as are $\{p_i^\bL\}$, it follows that 
each $s_k^\bL$ is a bijection of $\L$.
From the convergence assumption for the $b_{ij}$, it follows that $s_k^\bL$ is in fact a homeomorphism
for every $k\in [m]$.
\end{proof}

If we identify $H:=(\Homeo 2^\omega)_{(x_1,\dots,x_n)}$ with
$(\Homeo \L)_{(p_1^\bL,\dots, p_n^\bL)}$, we have $(s_1^\bL,\dots,s_m^\bL)\in H^m$.
We will prove that this tuple is generic with respect to the diagonal conjugation action of $H$ on $H^m$.
We do so by following closely Kwiatkowska's argument for $\Homeo 2^\omega$ \cite[starting on page~1144]{Kw:GHC}.

Let $P$ be a finite clopen partition of $\L$. 
Let $S_P$ denote the set of all surjective relations $s$ on $P$ with the property that $(b,b)\in s$ for every $b\in P$ which contains at least one of the points $p_i^\bL$.
For $f\in H$ define
\[ f\restr_P := \{ (b,c) \in P^2 \st f(b)\cap c \neq \emptyset \}. \]
Then clearly $f\restr_P\in S_P$.

Let $\bs := (s_1^P,\dots, s_m^P)$ be an $m$-tuple of surjective relations on $P$. Define
\[
[P,\bs]:= \bigl\{ (f_1,\dots,f_m)\in H^m \st
f_i\restr_P=s_i^P \text{ for all } i\in [m]\bigr\}
\]
to be the set of tuples $f_1,\dots,f_m$ that act the same as $s_1^P,\dots,s_m^P$, respectively, on $P$.
Any such set $[P,\bs]$ is precisely the intersection of $H^m$ with the corresponding subset
of $(\Homeo 2^\omega)^m$ defined in \cite[page~1144]{Kw:GHC}. Note that $[P,\bs]$ is non-empty
if and only if $s_1^P,\dots,s_m^P \in S_P$.

For $P$ and $\bs$ as above, define
$\P_\bs\in\F_0^{(n)}$
with
\begin{itemize}
\item universe $P$,
\item $s_i^{\P_s}:=s_i^P$ for $i\in [m]$, and
\item  $p_i^{\P_s}$ the unique member of $P$ containing $p_i^\bL$ for $i\in [n]$.
\end{itemize}

\begin{lemma}[{cf. \cite[Lemma 4.11]{Kw:GHC}}]
\label{la:topb}
 The sets $[P,\bs]$ for finite clopen partitions $P$ of $\L$ and $\bs\in S_P^m$
 form a topological basis of $H^m$.
\end{lemma}

\begin{proof}
 This follows directly from \cite[Lemma 4.11]{Kw:GHC} since the topology of pointwise convergence on
 $H^m$ is the subspace topology induced by the topology of pointwise convergence on $(\Homeo L)^m$.
\end{proof}

\begin{lemma}[{cf. \cite[Proposition 4.12]{Kw:GHC}}]
\label{la:dense}
The orbit $(s_1^\bL,\dots,s_m^\bL)^H$ under the diagonal conjugation action of $H$ on $H^m$ is dense in~$H^m$.
\end{lemma}

\begin{proof}
For a finite clopen partition $P$ of $\L$ and $\bs\in S_P^m$, let
\[  D(P,\bs):=\bigl\{ (f_1,\dots, f_m)^g\st (f_1,\dots, f_m)\in [P,\bs], g \in H
  \bigr\}. \]
Let $D$ be the intersection of all sets $D(P,\bs)$ for partitions $P$ of $\L$ and $\bs\in S_P^m$.
We claim that $(s_1^\bL,\dots,s_m^\bL)\in D$, and the result then follows from Lemma \ref{la:topb}.

For a finite clopen partition $P$ of $\L$ and $\bs\in S_P^m$, there exist $\A\in\F^{(n)}$ and an epimorphism $\phi\colon\A\rightarrow \P_s$ by  Lemma~\ref{la:Focoinit}.
By the Projective Universality \ref{it:L1} of $\bL$ there exists an epimorphism
$\psi\colon \bL\rightarrow \A$.
Let $Q$ be the partition of $\bL$ determined by the kernel of the epimorphism $\phi\psi$.
Then 
\[
\P\cong (Q; s_1^\bL\restr_Q,\dots,s_m^\bL\restr_Q,p_1^\bL\restr_Q,\dots,p_n^\bL\restr_Q).
\]
Here $p_i^\bL\restr_Q$ denotes the unique block in $Q$ that contains $p_i^\bL$.
Clearly we have $g\in H$ that induces an isomorphism between these two structures.
Then $(s_1^\bL,\dots,s_m^\bL)^g\in [P,\bs]$, proving that
$(s_1^\bL,\dots,s_m^\bL)\in D(P,\bs)$.
This completes the proof of the claim and of the lemma.
\end{proof}

\begin{lemma}[{cf. \cite[Proposition 4.13]{Kw:GHC}}]
\label{la:Gdelta}
The orbit $(s_1^\bL,\dots,s_m^\bL)^H$ is the intersection of countably many open sets in $H^m$.
\end{lemma}

\begin{proof}
For a tuple $\ff=(f_1,\dots,f_m)\in H^m$, let
\[
\bL_\ff:=(\L; f_1,\dots,f_m,p_1^\bL,\dots,p_n^\bL).
\]
By Theorem \ref{thm:IrwSol} and Proposition \ref{pr:EP} the orbit $(s_1^\bL,\dots,s_m^\bL)^H$ is precisely the set of all $\ff\in H^m$ such that 
$\bL_\ff$ is isomorphic to $\bL$, i.e.
the set of all $\ff\in H^m$ such that
$\bL_\ff$ satisfies Projective Universality \ref{it:L1}, the Factoring Property \ref{it:L2} and the Extension Property \ref{it:L4}. 
In what follows we will show how to express the set of all $\ff\in H^m$ satisfying each one of these properties in turn as a countable intersection of open sets.

\noindent
\textbf{Projective Universality \ref{it:L1}.}
Begin by noting that for $\A\in\F^{(n)}$ the set
\[
U_\A:=\bigl\{ \ff\in H^m\colon \exists \phi \st  \bL_\ff\twoheadrightarrow \A\bigr\}
\]
is open. This is because 
\begin{align*} U_\A= \bigcup \bigl\{ [P,\bs] \st & P \text{ is a clopen partition of } \L,\, \bs\in S_P^m,\, \P_\bs \cong \A \bigr\}
 \end{align*}
is a union of open sets by Lemma~\ref{la:topb}.
The set of all $\ff\in H^m$ such that $\bL_\ff$ satisfies Projective Universality is the countable intersection $\bigcap_{\A\in \F^{(n)}}U_\A$.

\noindent
\textbf{Factoring Property \ref{it:L2}.}
Consider any continuous mapping $u$ from $\L$ to some finite discrete space.
Let $T_u$ be the set of all $\ff\in H^m$ for which there is a finite partition of $\L$ into clopens refining the partition
induced by $\ker(u)$ such that the corresponding quotient of $\bL_\ff$ is in $\F^{(n)}$.
Note that
\[
 T_u=\bigcup  \bigl\{ [P,\bs] \st P \text{ refines } \ker(u),\, \bs\in S_P^m,\, \P_\bs \cong\A\in \F^{(n)} \bigr\}
\]
 is open.
The set of all $\ff\in H^m$ such that $\bL_\ff$ satisfies property \ref{it:L2} is $\bigcap_u T_u$.
This intersection is countable because there are only countably many partitions of $\L$ into finitely many clopens.
We note in passing that this argument also applies to the case $n=0$, but it was missed in the original proof in \cite{Kw:GHC} (see step (3) in the proof of \cite[Proposition 4.13]{Kw:GHC}).

\noindent
\textbf{Extension Property \ref{it:L4}.}
Let $\A,\B\in \F^{(n)}$, let $\phi_1\colon \L\rightarrow A$ be a continuous surjection, and
let ${\phi_2\colon \B\rightarrow \A}$ be an epimorphism.
Define the set
$E_{\phi_1,\phi_2}$ to consist of all
$\ff\in H^m$ satisfying the following property:
if $\phi_1$ is a homomorphism $\bL_\ff\rightarrow \A$, then there exists an epimorphism $\psi\colon \bL_\ff\rightarrow \B$ such that $\phi_2\psi=\phi_1$.
This is an open set because it is the union of all sets of the form $[P,\bs]$ where at least one of the
following is satisfied:
\begin{itemize}
\item
$P=\ker(\phi_1)$ and $\phi_1$ does not induce an isomorphism $\P_\bs\to\A$; or
\item
  $P=\ker(\psi)$ for some continuous surjection
  $\psi\colon \L\rightarrow B$ which induces an isomorphism $\P_\bs\to\B$ and satisfies ${\phi_2\psi=\phi_1}$.
\end{itemize}
The set of all $\ff\in H^m$ such that $\bL_\ff$ satisfies the Extension Property is $\bigcap_{\phi_1,\phi_2} E_{\phi_1,\phi_2}$.
The intersection is countable because there are only countably many choices for $\A$ and $\B$,
countably many choices for $\phi_1\colon\L\rightarrow A$, and
finitely many choices for $\phi_2\colon\A\rightarrow\B$.

This concludes the proof that
$\{ \ff\in H^m \st \bL_\ff\cong \bL \}$ is an intersection of countably many open sets in $H^m$
and of the lemma.
\end{proof}

 Finally we are ready to prove the main results of this section.

\begin{proof}[Proof of Theorem~\ref{thm:H}]
 Immediate from Lemmas \ref{la:dense} and \ref{la:Gdelta}.
\end{proof}

\begin{proof}[Proof of Theorem~\ref{thm:H2}]
 By Lemma~\ref{la:LCantor} we may assume that the projective Fra{\"\i}ss\'e limit
\[\bL = (2^\omega;s_1^\bL,\dots,s_m^\bL,p_1^\bL,\dots,p_n^\bL) \]
 of $\F^{(n)}$ has universe the Cantor space $2^\omega$ and $p_i^\bL = x_i$ for $i\in [n]$.
 By Lemma~\ref{la:graphhomeo} we have $s_1^\bL,\dots,s_m^\bL\in H$.
 Note that $\alpha\colon \bL\to (2^\omega;h_1,\dots,h_m$, $x_1,\dots,x_n)$ is an isomorphism if and only if
 $\alpha\in H$ and $(h_1,\dots,h_m)^\alpha = (s_1^\bL,\dots,s_m^\bL)$.
 The latter condition is equivalent to $(h_1,\dots,h_m)$ being generic by Lemmas \ref{la:dense} and
 \ref{la:Gdelta}.
\end{proof}

\section{Labellings and quotient property}
\label{sec:qp}

To obtain a proof of Theorem~\ref{thm:DAG} in the non-abelian case from Theorem~\ref{thm:H}, we need to understand
the diagonal conjugation action in $\Aut\ABxe$.
For this in turn, we will need some combinatorial observations concerning the elements of $\F_0^{(n)}$,
which we establish in this section.

 Throughout this section let $\sigma$ be the signature $\{s_1 ,\dots,s_m\}$
 consisting of $m$ binary relation symbols.
Let $\A,\B$ be $\sigma$-structures and let ${\phi\colon\B\rightarrow\A}$ be an epimorphism.
Let $T$ be a group. The reader may think of it as 
(a subgroup of) $\Aut\AA$.
Let $\lambda\colon A\to T^m$ be a \emph{labelling} of $\A$ by elements of $T^m$,
and let $\mu\colon B\to T$ be a labelling of $\B$ by $T$.
We say that $\phi,\lambda,\mu$ satisfy the \emph{quotient property} (QP) if
\begin{equation}
\tag{QP} \label{eq:QP}
\forall i\in [m]\ \forall (x,y)\in s_i^\B\colon \mu(x)^{-1} \mu(y)=\lambda\phi(y)_i.
\end{equation}
 Here $\lambda\phi(y)_i$ denotes the $i$-th component of $\lambda\phi(y)\in T^m$.

\begin{figure}
\begin{tikzpicture}
\node (a) at (0,0)
         {\begin{tikzpicture} [decoration = {markings,
    mark = between positions 0.1 and 1 step 0.125 with {\arrow[>=stealth]{>}}
  }
  ]
  \draw[postaction = decorate] (0, 0) circle [radius = 2cm];
  \path (3:2cm) node[fill=white]{$a_p=b_1$}
  (45:2cm) node[fill=white]{$a_1$}
  (90:2cm) node[fill=white]{$a_2$}
  (315:2cm) node[fill=white]{$a_{p-1}$};
\end{tikzpicture}
};
\node (b) at (4,0)
         {\begin{tikzpicture} [decoration = {markings,
    mark = between positions 0.15 and 1 step 0.33 with {\arrow[>=stealth]{>}}
  }
  ]
  \draw[postaction = decorate] (0, 0) -- (3.5,0);
 \path (0:0.8cm) node[fill=white]{$b_2$};
\end{tikzpicture}
};
\node (c) at (7.5,0)
 {\begin{tikzpicture}[decoration = {markings,
     mark = between positions 0.3 and 1 step 0.165 with {\arrow[>=stealth]{>}}
   }
   ]
   \draw[postaction = decorate] (0, 0) circle [radius = 1.5cm];
   \path (181:1.5cm) node[fill=white]{$b_q=c_1$}
   (240:1.5cm) node[fill=white]{$c_2$}
   (300:1.5cm) node[fill=white]{$c_3$}
   (120:1.5cm) node[fill=white]{$c_{r}$};  
\end{tikzpicture}   
};
\end{tikzpicture}
\caption{Spiral digraph $\S(p,q,r)$.}
\label{fig:spiral}
\end{figure}

We now introduce what Kwiatkowska \cite[Section 3]{Kw:GHC} calls \emph{spiral structures}
in a slightly different notation (see Figure~\ref{fig:spiral}).
For $p,q,r\in\N$ with $p,r>1$, we let $\S=\S(p,q,r):=(S(p,q,r);s^\S)$ be the digraph
\begin{itemize}
\item
 with $p+q+r-2$ vertices
  $$S(p,q,r) :=\{a_i\colon i\in [p]\}\cup \{ b_i\colon i\in [q]\}\cup\{ c_i\colon i\in [r]\},$$
  where $a_p=b_1$, $b_q=c_1$, and 
\item
  edge relation
  \begin{align*} s^\S:= & \bigl\{ (a_i,a_{i+1})\st i\in [p-1]\bigr\}  \cup
                        \bigl\{ (b_i,b_{i+1})\st i\in [q-1]\bigr\}  \cup \\
                        & \bigl\{ (c_i,c_{i+1})\st i\in [r-1]\bigr\}\cup\bigl\{ (a_p,a_1),(c_r,c_1)\bigr\}.
  \end{align*}
\end{itemize}

The next lemma establishes that for any labelling of a spiral there is a spiral preimage satisfying \eqref{eq:QP}.

\begin{lemma}
\label{la:spiralQP}
Let $m,p,q,r\in\N$ with $p,r>1$, and let $T$ be a group of exponent $t\in\N$.
For any labelling $\lambda\colon \S(p,q,r)\rightarrow T^m$ there exists a labelling
$\mu\colon \S(tp,q,tr)\rightarrow T$ and an epimorphism $\phi\colon \S(tp,q,tr)\rightarrow \S(p,q,r)$
satisfying \eqref{eq:QP}.
Furthermore, for any $x_0\in S(tp,q,tr)$ and any $\alpha\in T$, the labelling $\mu$ can be chosen so that
$\mu(x_0)=\alpha$.
\end{lemma}

\begin{proof}
 Define
\begin{equation}
\label{eq:Sepi}
\phi\colon S(tp,q,tr)\rightarrow S(p,q,r) \text{ by } 
\begin{cases}
  \phi(a_i) := a_j  & \text{for } i\in [tp], j\in [p], i\equiv_p j, \\
  \phi(b_i) := b_i & \text{for } i\in [q], \\
  \phi(c_i) := c_j  & \text{for } i\in [tr], j\in [r], i\equiv_r j. 
\end{cases}
\end{equation}
It is obvious that $\phi$ is a digraph epimorphism.

The following is the key observation:
given $\lambda$ on $S(p,q,r)$ and either $\mu(x)$ or $\mu(y)$ for an edge $(x,y)$ in $\S(tp,q,tr)$,
the quotient property uniquely determines the other.
Now consider the directed path
\[
a_1\rightarrow \dots\rightarrow a_{tp}=b_1\rightarrow \dots \rightarrow b_q=c_1\rightarrow\dots\rightarrow c_{tr}
\]
in $\S(tp,q,tr)$. Starting from $x_0$ and $\mu(x_0)=\alpha$, and applying this observation forward and backward along the path, defines $\mu\colon S(tp,q,tr)\rightarrow T$,
which gives \eqref{eq:QP} \emph{for every $(x,y)$ belonging to the path}.

The only edges of $\S(tp,q,tr)$ not belonging to the path are $(a_{tp},a_1)$, $(c_{tr},c_1)$.
We claim that \eqref{eq:QP} holds for these two edges as well. 
Indeed, by \eqref{eq:Sepi} and \eqref{eq:QP} along $a_1\rightarrow\dots\rightarrow a_{tp}$, we have
\begin{equation}
\label{eq:QPa}
\mu(a_{up+i-1})^{-1}\mu(a_{up+i})=\lambda(a_i) \quad \text{for }
u\in [0,t-1],\ i\in [p],\ up+i>1.
\end{equation}
Multiplying all these equalities and using Lagrange's Theorem yields
\[
\mu(a_1)^{-1}\mu(a_{tp})=\lambda(a_2)\dots\lambda(a_p) \bigl(\lambda(a_1)\dots\lambda(a_p)\bigr)^{t-1}=\lambda(a_1)^{-1}.
\]
Hence $\mu(a_{tp})^{-1}\mu(a_1)=\lambda(a_1)$, which is precisely \eqref{eq:QP} for $(a_{tp},a_1)$.
The proof for $(c_{tr},c_1)$ is analogous, completing the proof of the claim and the lemma.
\end{proof}

We now move on to arbitrary surjective $\sigma$-structures.

\begin{lemma}
\label{la:surjQP}
Let $T$ be a group of exponent $t \in\N$.  
For any finite surjective $\sigma$-structure $\A=(A;s_1^\A,\dots,s_m^\A)$ and any labelling $\lambda \colon A\rightarrow T^m$
there exist a finite surjective $\sigma$-structure $\B$, an epimorphism $\phi\colon \B\rightarrow\A$, and a labelling $\mu\colon B\rightarrow T$ satisfying \eqref{eq:QP}.
\end{lemma}

\begin{proof}
We will build the domain of $\B$ and its relations in stages.

Fix $i\in [m]$.
By \cite[Lemma 4.8]{Kw:GHC}, the digraph reduct $(A;s_i^\A)$ of $\A$ has a preimage $\C_i$
that is a disjoint union of spirals.
For every $a\in A$, label all preimages of $a$ in $\C_i$ by $\lambda(a)$.
By Lemma \ref{la:spiralQP}, for each spiral $\S(p,q,r)$ in $\C_i$,
the spiral $\S(tp,q,tr)$ is a preimage of 
$\S(p,q,r)$, and it has a labelling $\mu\colon S(tp,q,tr)\rightarrow T$ such that \eqref{eq:QP} holds.
In fact, there are $|T|$ such labellings. Taking the disjoint union over all spirals $\S(tp,q,tr)$
and all their labellings thus constructed,
we obtain a structure $\B_i=(B_i; s_1^{\B_i},\dots, s_m^{\B_i})$ such that 
\begin{itemize}
\item
$s_i^{\B_i}$ is surjective and
\item
$s_j^{\B_i}=\emptyset$ for $j\neq i$,
\end{itemize}
and a labelling $\mu_i\colon B_i\rightarrow T$.
Moreover, the union of the homomorphisms from the spirals $\S(tp,q,tr)$ to $\C_i$ composed with the
epimorphism $\C_i\to(A;s_i^\A)$ yields an epimorphism $\phi_i\colon \B_i\rightarrow \A$
such that
\begin{itemize}
\item
 $\phi_i,\lambda, \mu_i$ satisfy \eqref{eq:QP} and
\item
for every $a\in A$ and every $\alpha\in T$ there exists $b\in B_i$ such that $\phi_i(b)=a$ and $\mu_i(b)=\alpha$.
\end{itemize}

Now, take the disjoint union over all $i\in [m]$ to obtain the structure
\[ \B := (B; s_1^{\B_1},\dots,s_m^{\B_m}) \text{ with } B := \bigcup_{i\in [m]} B_i, \]
the epimorphism $\phi:=\bigcup_{i\in [m]} \phi_i\colon \B\rightarrow \A$ and the labelling
$\mu:=\bigcup_{i\in [m]}\mu_i\colon$ $B\rightarrow T$ such that the following hold:
\begin{itemize}
\item
$s_i^\B$ is surjective on $B_i$ (and empty elsewhere);
\item
 $\phi,\lambda,\mu$ satisfy \eqref{eq:QP};
\item
for every $a\in A$, $\alpha\in T$ and $i\in [m]$ there exists $b\in B_i$ such that $\phi(b)=a$ and $\mu(b)=\alpha$.
\end{itemize}

This is not yet the structure we are aiming for: the issue is that the relations $s_i^\B$ are not
surjective for $m>1$.
We will now successively add some extra pairs to each $s_i^\B$ in such a way that $\phi$
remains a homomorphism, \eqref{eq:QP} continues to hold for $\phi,\lambda,\mu$, and at the end of the
process all $s_i^\B$ are surjective.

Let $i,j\in [m]$ be distinct and $b\in B_i$.
Since $s_j^\A$ is surjective, there exists $a'\in A$ such that $(\phi(b),a')\in s_j^\A$.
Let $b'\in B_j$ be such that $\phi(b')=a'$ and $\mu(b')=\mu(b)(\lambda(a'))_j$.
Such a $b'$ exists by the construction of $\B_j$.
Add the pair $(b,b')$ to $s_j^\B$.
It is straightforward to verify that with this choice of $b'$ the mapping $\phi$ remains a homomorphism
and $\phi,\lambda,\mu$ still satisfy \eqref{eq:QP}.
The analogous choice, starting with $(a'',\phi(b))\in s_j^\A$ lets us add another pair $(b'',b)$ to $s_j^\B$.
Repeating this for all distinct $i,j\in [m]$ and $b\in B_i$ results in all $s_j^\B$ becoming surjective
and completes the proof.
\end{proof}

\section{Filtered Boolean powers}
\label{sec:fbp}

In this section we prove our main result Theorem \ref{thm:DAG} for the non-abelian case.
Throughout let $\AA$ be a fixed finite simple non-abelian Mal'cev algebra, let $n\geq 0$,
let $\be := (e_1,\dots,e_n)$ be a (possibly empty) tuple of singleton subalgebras of~$\AA$, and
let $\bx := (x_1,\dots,x_n)$ be a tuple of distinct points in $2^\omega$.
We denote the filtered Boolean power $\ABxe$ by $\DD$ and its automorphism group $\Aut\DD$ by $G$.
We can assume without loss of generality that $e_1,\dots,e_n$ are all in distinct $\Aut\AA$-orbits.
If not, we can just take a subset of orbit representatives without changing the isomorphism type of $\DD$
by  \cite[Corollary 2.8]{MR:FBP}.
By \cite[Theorem 2.5]{MR:FBP} the group $G$ splits into a semidirect product
\[ G=K\rtimes H, \]
where $H$ is isomorphic to the pointwise stabiliser $(\Homeo 2^\omega)_\bx$
and the normal subgroup $K$ is a closure of the filtered Boolean power
$(\Aut\AB)^\bx_{\mathbf{1}}$
in $G$, where $\mathbf{1}$ is the $n$-tuple $(1,\dots,1)$ for $1$ the identity of $\Aut\AA$.
We recall from~\cite[beginning of Section 2.3]{MR:FBP} that
\begin{equation*} \label{eq:H1}
H=\bigl\{ \overline{\psi} \st \psi \in (\Homeo 2^\omega)_\bx\bigr\},
\end{equation*}
where
\begin{equation}
\label{eq:H2}
\overline{\psi}\colon D\rightarrow D,\ f\mapsto f\psi^{-1},
\end{equation}
and $\psi\mapsto\overline{\psi}$ is an isomorphism $(\Homeo 2^\omega)_\bx\rightarrow H$.
 Recall that $2^{\omega\circ}$ denotes the punctured space $2^\omega \setminus\{x_1,\dots,x_n\}$.
By \cite[Theorem 2.5]{MR:FBP} the kernel $K$ is isomorphic to 
\begin{multline}
\label{eq:K1}
K' := \bigl\{ \psi\colon 2^{\omega\circ}\rightarrow \Aut\AA \st \psi \text{ continuous},\\
\psi^{-1}\bigl((\Aut\AA)_{e_i}\bigr)\cup\{x_i\}\text{ open in } 2^\omega, \ \forall i\in [n]\bigr\}.
\end{multline}
The isomorphism given in \cite[Theorem 2.5]{MR:FBP} is
\begin{equation}
\label{eq:K2}
K\to K',\ \phi \mapsto [2^{\omega\circ}\rightarrow \Aut\AA,\ x\mapsto \phi_x],
\end{equation}
where $\phi_x\in \Aut\AA$ is uniquely defined by
\begin{equation}
\label{eq:K3}
\phi_x(f(x))=\bigl[\phi(f)\bigr](x) \text{ for all } f\in D,\ x\in 2^{\omega\circ}.
\end{equation}
In what follows we will prefer to use the inverse of this isomorphism, namely
\begin{equation}
\label{eq:K4}
K'\rightarrow K,\ \psi\mapsto \widehat{\psi},
\text{ where }
\bigl[\widehat{\psi}(f)\bigr](x)=\bigl[\psi(x)\bigr](f(x)) \text{ for } x\in 2^{\omega\circ}.
\end{equation}

The basic idea behind the proof of Theorem \ref{thm:DAG} is as follows.
By Theorem~\ref{thm:H}, $H\cong (\Homeo 2^\omega)_\bx$ has ample generics.
We will show that any generic in $H^m$ is also generic in $G^m$.
More precisely, let $\hh:=(h_1,\dots,h_m)$ be a generic $m$-tuple for $(\Homeo 2^\omega)_\bx$ throughout this section.
We will prove that $\overline{\hh}:=(\overline{h}_1,\dots,\overline{h}_m)$ is a generic $m$-tuple for $G$.

 Lemma \ref{la:surjQP} is the key combinatorial fact behind the following somewhat unexpected result,
 which asserts that every translate of $\overline{\hh}$  by a tuple from $K^m$ can be expressed as a conjugate
 of $\overline{\hh}$ by an element of~$K$.

\begin{lemma}
\label{la:transconj}
For every $\ba\in K^m$ there exists $c\in K$ such that $\ba \overline{\hh}=\overline{\hh}^c$.
\end{lemma}

\begin{proof}
 We will first prove the lemma in the case $n=0$ when there are no filtering points, i.e.~we are dealing with the
 `full'/`unfiltered' Boolean power $\DD := \AB$.
 For $n>0$ we will then use the trick from Section \ref{sec:X0} of partitioning $2^{\omega\circ}=2^\omega\setminus \{ x_1,\dots,x_n\}$ into $n$ infinite sequences of clopens converging to $x_1,\dots, x_n$, respectively.
\medskip

\noindent
{\bf Case \boldmath{$n=0$}.}
Let $\ba = (a_1,\dots,a_m)\in K^m$.
From the description \eqref{eq:K1}--\eqref{eq:K4} of the kernel $K$, we have $k\in\N$ and a partition
$P=\{b_1,\dots,b_k\}$ of $2^\omega$ into clopens such that 
each $a_i$ acts as a fixed automorphism $\alpha_{ij}\in\Aut\AA$ on each $b_j$, i.e.
\begin{equation}
\label{eq:trco1}
\bigl[a_i(f)\bigr](x)=\alpha_{ij}f(x)\quad \text{for all } i\in [m],\ f\in D,\ x\in b_j, j\in [k].
\end{equation}
Recall that $(2^\omega;h_1,\dots,h_m)$ is (isomorphic to) the projective Fra{\"\i}ss\'e limit of
$\F = \F^{(0)}$ by Theorem~\ref{thm:H2}.
Let $\P=(P; h_1^\P,\dots,h_m^\P)$ be the (surjective) structure induced by $\hh$ on $P$; in other words,
$\P$ is the image of $(2^\omega;h_1,\dots,h_m)$ via the natural projection $\phi_1\colon 2^\omega\rightarrow P$.
Without loss of generality we can assume that in fact $\P\in \F$.
Indeed, if it is not, we can use the Factoring Property \ref{it:L2} to replace $\P$ by a preimage in $\F$
through which $\phi_1$ factors.
Note that this preimage induces a refinement of the partition $P$ of $2^\omega$ and each
$a_i$ still acts like a fixed automorphism on each of its blocks.

We can consider the automorphisms $\alpha_{ij}^{-1}$ as a labelling $\lambda\colon P\rightarrow (\Aut\AA)^m$.
By Lemma \ref{la:surjQP} there exist a surjective structure $\Q=(Q;s_1^\Q,\dots, s_m^\Q)$, an epimorphism
$\phi_2\colon \Q\rightarrow \P$ and a labelling $\mu\colon Q\rightarrow \Aut\AA$ satisfying \eqref{eq:QP},
i.e.
\begin{equation}
\label{eq:trco2}
\mu(x)^{-1}\mu(y)=\alpha_{ij}^{-1} \quad\text{for } (x,y)\in s_i^\Q,\  i\in [m],\ \phi_2(y)= b_j,\ j\in [k].
\end{equation}
Furthermore, the values $\mu(x)$ can all be chosen to belong to the subgroup of $\Aut\AA$ generated by the $\alpha_{ij}$.
Again, we may assume that $\Q\in\F$;
otherwise we replace it using \ref{it:L2} with a preimage through which $\phi_2$ factors and label the preimage
of any point from the original $Q$ by the label of that point.
Now, using the Extension Property \ref{it:L4}, there exists an epimorphism
$\psi\colon (2^\omega;h_1,\dots,h_m)\rightarrow \Q$
such that $\phi_1=\phi_2\psi$.

Define an automorphism $c\in K$ which is constant on the refinement of $P$ that is induced by $\Q$ via
\begin{equation}
\label{eq:defc}
\bigl[c(f)\bigr](x) :=\bigl[\mu\psi(x)\bigr]\, \bigl(f(x)\bigr)\quad \text{for } f\in D,\ x\in 2^\omega.
\end{equation}
 We claim that $\ba\overline{\hh} = \overline{\hh}^c$ as required.

 For verifying this let $i\in [m]$, $f\in D$, $x\in 2^\omega$, and 
\begin{equation}
\label{eq:qr}
q:= \psi(x),\  r:=\psi h_i^{-1}(x),\ b_j:=\phi_2(r).
\end{equation}
 Then
\begin{align*}
 \bigl[c^{-1}\overline{h}_i^{-1}c(f)\bigr](x) &=
\bigl[\mu(q)^{-1}\bigr]\bigl( \bigl[\overline{h}_ic(f)\bigr](x)\bigr)&&\text{by \eqref{eq:defc}, \eqref{eq:qr}}
\\
&=\bigl[\mu(q)^{-1}\bigr] \Bigl( \bigl[c(f)\bigr]\bigl(h_i^{-1}(x)\bigr)\Bigr)&&\text{by \eqref{eq:H2}}
\\
&=\bigl[\mu(q)^{-1}\bigr]\Bigl( \bigl[\mu(r)\bigr] \bigl(fh_i^{-1}(x)\bigr)\Bigr)&&\text{by \eqref{eq:defc}, \eqref{eq:qr}}\\
&=\bigl[\mu(q)^{-1} \mu(r)\bigr] \bigl(fh_i^{-1}(x)\bigr).
\end{align*}
Since $\psi$ is a homomorphism, we have 
\[
(r,q)=\bigl(\psi h^{-1}(x),\psi(x)\bigr)\in s_i^\Q.
\]
Therefore, using \eqref{eq:trco2}, \eqref{eq:qr}, we have
\[
\bigl[\mu(q)^{-1} \mu(r)\bigr] \bigl(fh_i^{-1}(x)\bigr) =
\alpha_{ij}fh_i^{-1}(x).
\]
From \eqref{eq:qr} and $\phi_1=\phi_2\psi$, we obtain
$\phi_1(h_i^{-1}(x))=b_j$, i.e.~$h_i^{-1}(x)\in b_j$.
Hence
\[
\alpha_{ij}\bigl(fh_i^{-1}(x)\bigr)=\bigl[ a_i(f)\bigr] \bigl(h_i^{-1}(x)\bigr)=\bigl[[a_i\overline{h}_i](f)\bigr](x),
\]
by \eqref{eq:H2}. Putting the above calculations together yields
\[
[c^{-1}\overline{h}_i^{-1}c(f)](x)=\bigl[[a_i\overline{h}_i](f)\bigr](x)
\]
for all $i\in [m]$, $f\in D$, $x\in 2^\omega$. Therefore $c^{-1}\overline{\hh}c=\ba\overline{\hh}$,
as required.
\medskip

\noindent
{\bf Case \boldmath{$n>0$}.}
Using Lemma \ref{la:clcoco}, partition $ 2^{\omega\circ} = 2^\omega\setminus \{ x_1,\dots, x_n\}$ into clopens $b_{ij}$ for
$i\in [n]$,
$j\in\N$, which are unions of connected components of $(2^\omega;h_1,\dots,h_m)$ and such that each sequence $(b_{ij})_j$
converges to $x_i$ for $i\in [n]$.

Let $b:=b_{ij}$.
By Lemma \ref{la:MLk} the structure $(b; h_1\restr_b,\dots,h_m\restr_b)$ is the projective Fra\"iss\'e limit
of $\F$.
Since $b$ is homeomorphic to $2^\omega$, the Stone space of $\BB$, by the case $n=0$ there exists
a continuous $c_{ij}\colon b \to\Aut\AA$
such that $(\ba\restr_b)(\overline{\hh}\restr_b)=$ $(\overline{\hh}\restr_b)^{\widehat{c}_{ij}}$;
 see \eqref{eq:K4} for the definition of $\widehat{c}_{ij}$.
Furthermore, if  $a_x\in (\Aut\AA)_{e_i}$ for all $x\in b$,  then $c_{ij}$ can be chosen to also satisfy
$c_{ij}(x)\in (\Aut\AA)_{e_i}$ for all $x\in b$;
see the remark following \eqref{eq:trco2}, and also see \eqref{eq:K3} for the definition of $a_x$.

Now, let $c$ be the union of all the $c_{ij}$ over all $b_{ij}$ partitioning $2^{\omega\circ}$.
We claim that $c\in K'$.
Indeed, the set $\{ x\in 2^{\omega\circ} \colon c(x)=\alpha\}$ is open in $2^{\omega\circ}$
for every $\alpha\in \Aut\AA$, 
because it is the  union of clopen sets $\{x\in b_{ij}\colon c_{ij}(x)=\alpha\}$.
Furthermore, since $\ba\in K^m$, and since the sequence $(b_{ij})_j$ converges to $x_i$, it follows that 
$(a_k)_x\in(\Aut\AA)_{e_i}$ for all $k\in [m]$, $x\in b_{ij}$ and sufficiently large $j$.
Therefore 
$c(b_{ij})\subseteq (\Aut\AA)_{e_i}$ for all sufficiently large $j$.
With \eqref{eq:K1} it now follows that $c\in K'$, completing the proof of the claim.
Then $\widehat{c}\in K$ and clearly $\ba\hh=\hh^{\widehat{c}}$. This completes the proof of case $n>0$
and of the lemma.
\end{proof}

\begin{lemma}
\label{la:concos}
$\overline{\hh}^G= K^m\, \overline{\hh}^H$.
\end{lemma}

\begin{proof}
($\subseteq$) 
Let $g\in G$ be arbitrary, and write it as $g=dc$ with $c\in K$, $d\in H$.
Let $\bc:=(c,\dots,c)\in K^m$. Then
\[
\overline{\hh}^g=\overline{\hh}^{dc}=\bc^{-1}\cdot\overline{\hh}^d\bc (\overline{\hh}^d)^{-1}\cdot \overline{\hh}^d\in K^m\, \overline{\hh}^H
\]
because $K$ is normal in $G$.

\noindent
($\supseteq$) This follows from Lemma \ref{la:transconj}. 
\end{proof}

 Finally we can complete the proof of our main result.

\begin{proof}[Proof of Theorem \ref{thm:DAG} in the non-abelian case]
As above, let $\hh$ be a generic $m$-tuple for the group $(\Homeo 2^\omega)_\bx$. 
We will show that $\overline{\hh}^G$ is dense and the  intersection of countably many open sets in $G^m$.
\medskip

\noindent
{\bf \boldmath{$\overline{\hh}^G$} is dense in \boldmath{$G^m$}.}
Consider a non-empty open set $O_G$ in $G^m$.
Since the topology on $G^m$ is the product topology induced by the topology of pointwise convergence in $G$, 
it follows that $O_G$ must contain an open set of the form $\bg (G_\ff)^m$ for some $\bg \in G^m$, $k\in\N$ and
$\ff \in D^k$.
Write $\bg=\ba\uu$ with $\ba\in K^m$, $\uu\in H^m$.
The set $O_H:=\uu (H_\ff)^m$ is open in $H^m$ and satisfies $\ba O_H\subseteq O_G$.
Since $\overline{\hh}$ is a generic $m$-tuple for $H$, it follows that $\overline{\hh}^H \cap O_H \neq\emptyset$.
But then, using Lemma \ref{la:concos}, we have
\[
\emptyset\neq \ba\overline{\hh}^H\cap \ba O_H\subseteq \overline{\hh}^G\cap O_G.
\]
 Hence $\overline{\hh}^G$ intersects any nonempty open set and is dense in $G^m$.

\noindent
{\bf \boldmath{$\overline{\hh}^G$} is the intersection of countably many open sets.}
We first claim that
\begin{equation} \label{eq:KmO}
 K^m O \text{ is open in } G^m \text{ for every open set } O \text{ in } H^m.
\end{equation}  
It is sufficient to prove this for $O$ of the form
\[
O=\Bigl\{ \overline{\uu}\colon \uu\in \bigl((\Homeo 2^\omega)_{(b_1,\dots, b_k)}\bigr)^m\Bigr\},
\]
where $2^\omega=b_1\dotcup\dots\dotcup b_k$ is a partition into clopens.
Let $\ff$ be the tuple of all $f\in D$ which are constant on all $b_1,\dots, b_k$.
We will prove that $K^m O$ is open by proving that
\[
K^m O=K^m (G_\ff)^m.
\]
The direct inclusion ($\subseteq$) is immediate from $O\subseteq (G_\ff)^m$.
For the reverse inclusion ($\supseteq$) let $\bg\in (G_\ff)^m$ be arbitrary, and write $\bg=\ba\uu$ with
$\ba\in K^m$, $\uu\in H^m$.
As for \cite[(5.5)]{MR:FBP} one can show that $\uu\in O$, using that $e_1,\dots, e_n$ are in distinct
$\Aut\AA$-orbits.
Hence $K^m\bg=(K^m\ba)\uu\subseteq K^m O$, as required, and~\eqref{eq:KmO} is proved.

Since $\overline{\hh}$ is a generic $m$-tuple for $H$ by assumption, there exist countably many open subsets $O_i$
of $H^m$ for $i\in\N$ such that
\[ \overline{\hh}^H = \bigcap_{i\in\N} O_i. \]
Using Lemma~\ref{la:concos} and that $H$ is a complement for $K$ in $G$, we see that
\[ \overline{\hh}^G = K^m \overline{\hh}^H = K^m  \bigcap_{i\in\N} O_i =  \bigcap_{i\in\N} K^m O_i, \]
a countable intersection of open subsets of $G^m$ by~\eqref{eq:KmO}.

Thus $\overline{\hh}$ is a generic $m$-tuple for $G$ and the theorem is proved.
\end{proof}

\section{Open Questions}
\label{sec:qus}

 For $\AA$ a finite simple non-abelian group, let $\FF$ be the free group of countable rank in the
 variety generated by $\AA$. As observed in~\cite[p.~367-8]{BE:SIP} and \cite[p.~201]{BG:ARFG},
 $\FF$ has a normal subgroup
 $N$ isomorphic to the filtered Boolean power $(\AA^{\BB})^x_1$ with quotient $\FF/N$ the free algebra
 of countable rank in the variety generated by all proper subgroups of~$\AA$.
 Bryant and Evans~\cite{BE:SIP} showed that $\Aut(\AA^{\BB})^x_1$ has the small index property.
 Their question whether the same is true for the free group $\FF$ remains open.
 The same question is also of interest for $\AA$ an arbitrary finite simple non-abelian Mal'cev algebra
 and $\FF$ the free algebra of countable rank in the variety $\AA$ generates.
 Specifically, does $\Aut\FF$ have the small index property and ample generics?

\section*{Acknowledgments}

 We gratefully acknowledge the support of Billie, Kathryn and Martin in whose garden these results have been proved.

\bibliographystyle{plain}

\begin{thebibliography}{99}
\bibitem{Ap:BPG}
A.B.~Apps.
\newblock Boolean powers of groups.
\newblock {\em Math. Proc. Cambridge Philos. Soc.}, 91(3):375--395, 1982.

\bibitem{AK:TRA}
R.F.~Arens and I.~Kaplansky.
\newblock Topological representation of algebras.
\newblock {\em Trans. Amer. Math. Soc.}, 63:457--481, 1948.

\bibitem{BE:SIP}
R.M.~Bryant and D.M.~Evans.
\newblock The small index property for free groups and relatively free groups.
\newblock {\em J. London Math. Soc. (2)}, 55(2):363--369, 1997.

\bibitem{BG:ARFG}
R.M.~Bryant and J.R.J.~Groves.
\newblock On automorphisms of relatively free groups.
\newblock {\em J. Algebra}, 137(1):195--205, 1991.

\bibitem{Bu:BP}
S.~Burris.
\newblock Boolean powers.
\newblock {\em Algebra Universalis}, 5(3):341--360, 1975.

\bibitem{BS:CUA}
S.~Burris and H.P.~Sankappanavar.
\newblock {\em A course in universal algebra}.
\newblock Springer, New York Heidelberg Berlin, 1981. \newblock Available from

  \verb+www.math.uwaterloo.ca/~snburris/htdocs/UALG/univ-algebra2012.pdf+.

\bibitem{DH:GAG}
M.~Droste and C.~Holland.
\newblock Generating automorphism groups of chains.
\newblock {\em Forum Math.}, 17(4):699--710, 2005.

\bibitem{Ev:SSI}
D.M.~Evans.
\newblock Subgroups of small index in infinite general linear groups.
\newblock {\em Bull. London Math. Soc.}, 18(6):587--590, 1986.

\bibitem{Ev:EAC}
D.M.~Evans.
\newblock Examples of {$\aleph_0$}-categorical structures.
\newblock In {\em Automorphisms of first-order structures}, Oxford Sci. Publ.,
  pages 33--72. Oxford Univ. Press, New York, 1994.

\bibitem{Fo:GBT}
A.L.~Foster.
\newblock Generalized ``{B}oolean'' theory of universal algebras. {I}.
  {S}ubdirect sums and normal representation theorem.
\newblock {\em Math. Z.}, 58:306--336, 1953.

\bibitem{FM:CTC}
R.~Freese and R.N.~McKenzie.
\newblock {\em Commutator theory for congruence modular varieties}, volume 125
  of {\em London Math. Soc. Lecture Note Ser.}
\newblock Cambridge University Press, Cambridge, 1987.
\newblock Available from\\
  \verb+math.hawaii.edu/~ralph/Commutator/comm.pdf+.

\bibitem{Ho:MT}
W.~Hodges.
\newblock {\em Model theory}, volume~42 of {\em Encyclopedia of Mathematics and
  its Applications}.
\newblock Cambridge University Press, Cambridge, 1993.

\bibitem{HHLS:SIP}
W.~Hodges, I.~Hodkinson, D.~Lascar, and S.~Shelah.
\newblock The small index property for {$\omega$}-stable {$\omega$}-categorical
  structures and for the random graph.
\newblock {\em J. London Math. Soc. (2)}, 48(2):204--218, 1993.

\bibitem{Hr:EPI}
E.~Hrushovski.
\newblock Extending partial isomorphisms of graphs.
\newblock {\em Combinatorica}, 12:411--416, 1992.

\bibitem{IS:PFL}
T.~Irwin and S.~Solecki.
\newblock Projective Fra\"iss\'e{} limits and the pseudo-arc.
	{\em Trans. Amer. Math. Soc.}, 358:3077--3096, 2006.
	

 \bibitem{KR:TAG}
 A.S.~Kechris and C.~Rosendal.
 \newblock Turbulence, amalgamation, and generic automorphisms of homogeneous
   structures.
 \newblock {\em Proc. Lond. Math. Soc. (3)}, 94(2):302--350, 2007.

\bibitem{KT:GAU}
D.~Kuske and J.K.~Truss.
\newblock Generic automorphisms of the universal partial order.
\newblock {\em Proc. Amer. Math. Soc.}, 129(7):1939--1948, 2001.

\bibitem{Kw:GHC}
A.~Kwiatkowska.
\newblock The group of homeomorphisms of the {C}antor set has ample generics.
\newblock {\em Bull. Lond. Math. Soc.}, 44(6):1132--1146, 2012.

\bibitem{MR:CRN}
A.~Macintyre and J.G.~Rosenstein.
\newblock {$\aleph \sb{0}$}-categoricity for rings without nilpotent elements and for {B}oolean structures.
\newblock {\em J. Algebra}, 43(1):129--154, 1976.

\bibitem{Mac:SHS}
D.~Macpherson.
\newblock A survey of homogeneous structures.
\newblock {\em Discrete Math.}, 311(15):1599--1634, 2011.

\bibitem{MR:FBP}
 P.~Mayr and N.~Ru\v{s}kuc,
 Filtered Boolean powers of finite simple non-abelian Mal'cev algebras. Submitted. 
 \verb+arXiv:2404.17322+.


\bibitem{MMT:ALV1}
 R.N. McKenzie, G.F. McNulty, and W.F. Taylor.
 \newblock {\em Algebras, lattices, varieties, Volume {I}}.
 \newblock Wadsworth \& Brooks/Cole Advanced Books \& Software, Monterey, California, 1987.
  

\bibitem{Mo:HBA1}
J.D.~Monk and R.~Bonnet, editors.
\newblock {\em Handbook of {B}oolean algebras. {V}ol. 1}.
\newblock North-Holland Publishing Co., Amsterdam, 1989.

\bibitem{So:EPI}
S.~Solecki.
\newblock Extending partial isometries.
\newblock {\em Israel J. Math.}, 150:315--331, 2005.

\bibitem{Sz:CUA}
\'A. Szendrei.
\newblock {\em Clones in universal algebra}, volume~99 of {\em S\'eminaire de
  Math\'ematiques Sup\'erieures [Seminar on Higher Mathematics]}.
\newblock Presses de l'Universit\'e{} de Montr\'eal, Montreal, QC, 1986.

\bibitem{To:IGL}
V.A.~Tolstykh.
\newblock Infinite-dimensional general linear groups are groups of finite
  width.
\newblock {\em Sibirsk. Mat. Zh.}, 47(5):1160--1166, 2006.
\end{thebibliography}

\end{document}